\documentclass{amsart}
\usepackage[colorlinks=true,citecolor=blue,urlcolor=magenta]{hyperref}

\usepackage[ansinew]{inputenc}
\usepackage{graphicx}
\usepackage{amsmath}
\usepackage{amsthm,enumitem}
\usepackage{amssymb, color}%
\usepackage[numbers, square]{natbib}
\usepackage{mathrsfs}
\usepackage{bbm}
\usepackage{xcolor}

\newcommand{\R}{\mathbb{R}}
\newcommand{\E}{\mathbb{E}}
\newcommand{\N}{\mathbb{N}}

\newcommand{\LCM}{{\rm lcm}}
\newcommand{\GCD}{{\rm gcd}}
\newcommand{\Cov}{{\rm Cov}}
\newcommand{\Var}{{\rm Var}\,}

\renewcommand{\P}{\mathbb{P}}
\newcommand{\1}{\mathbbm{1}}

\renewcommand {\leq}{\leqslant}
\renewcommand {\geq}{\geqslant}

\newcommand{\mmp}{\mathbb{P}}

\newcommand{\me}{\mathbb{E}}

\newcommand{\mr}{\mathbb{R}}
\newcommand{\mn}{\mathbb{N}}

\newcommand{\eqdistr}{\stackrel{{\rm d}}{=}}

\newcommand{\todistr}{\overset{{\rm d}}{\underset{n\to\infty}\longrightarrow}}
\newcommand{\todistri}{\overset{{\rm d},\infty}{\underset{n\to\infty}\longrightarrow}}

\newcommand{\toprobab}{\overset{{\mmp}}{\underset{n\to\infty}\longrightarrow}}
\newcommand{\tofd}{\overset{{\rm f.d.d.}}{\underset{n\to\infty}\longrightarrow}}

\theoremstyle{plain}
\newtheorem{theorem}{Theorem}[section]
\newtheorem{lemma}[theorem]{Lemma}

\newtheorem{proposition}[theorem]{Proposition}

\theoremstyle{definition}

\theoremstyle{remark}

\begin{document}

\title[A Brownian limit for the LCM of a random {\scriptsize $m$}-tuple of integers]{A Brownian weak limit for the least common multiple of a random {\normalsize $m$}-tuple of integers}

\author{Dariusz Buraczewski}
\address{Dariusz Buraczewski, Mathematical Institute University of Wroclaw, Pl. Grunwaldzki 2/4
50-384 Wroclaw, Poland}\email{dbura@math.uni.wroc.pl}

\author{Alexander Iksanov}
\address{Alexander Iksanov, Faculty of Computer Science and Cybernetics, Taras Shev\-chen\-ko National University of Kyiv, 01601 Kyiv, Ukraine}
\email{iksan@univ.kiev.ua}

\author{Alexander Marynych}
\address{Alexander Marynych, Faculty of Computer Science and Cybernetics, Taras Shev\-chen\-ko National University of Kyiv, 01601 Kyiv, Ukraine}
\email{marynych@unicyb.kiev.ua}

\begin{abstract}
Let $B_n(m)$ be a set picked uniformly at random among all $m$-elements subsets of $\{1,2,\ldots,n\}$. We provide a pathwise construction of the collection 
$(B_n(m))_{1\leq m\leq n}$ and prove that the logarithm of the least common multiple of the integers in $(B_n(\lfloor mt\rfloor))_{t\geq 0}$, properly centered and normalized, converges to a Brownian motion when both $m,n$ tend to infinity. Our approach consists of two steps. First, we show that the aforementioned result is a consequence of a multidimensional central limit theorem for the logarithm of the least common multiple of $m$ independent random variables having uniform distribution on $\{1,2,\ldots,n\}$. Second, we offer a novel approximation of the least common multiple of a random sample by the product of the elements of the sample with neglected multiplicities in their prime decompositions.
\end{abstract}



\maketitle

\section{Introduction}

The analysis of divisibility properties of {\it random integers} is a classical problem in the probabilistic number theory going back to pioneer works by Dirichlet~\cite{Dirichlet:1849} and  Ces\`aro~\cite{Cesaro:1884,Cesaro:1885,Cesaro:1885b}. Among other results, Ces\`aro in~\cite{Cesaro:1885b} has proved that the expected least common multiple ($\LCM$) of two integers picked uniformly at random from the set $[n]:=\{1,2,\ldots,n\}$ is asymptotically, as $n\to\infty$, equal to their product multiplied by the constant $\zeta(3)/\zeta(2)$, where $\zeta$ is the Riemann zeta function. In the modern probabilistic language, this result can be stated as
\begin{equation}\label{eq:cesaro_seminal_result}
\lim_{n\to\infty}\E \left(\frac{\LCM(U_1^{(n)},U_2^{(n)})}{U_1^{(n)}U_2^{(n)}}\right)=\lim_{n\to\infty}\E \left(\frac{1}{\GCD(U_1^{(n)},U_2^{(n)})}\right)=\frac{\zeta(3)}{\zeta(2)},
\end{equation}
where $U_1^{(n)},U_2^{(n)}$ are independent copies of a random variable $U^{(n)}$ with distribution
\begin{equation}\label{eq:uniform_distribution}
\P\{U^{(n)}=k\}=1/n,\quad k\in [n],
\end{equation}
and $\GCD$ denotes the greatest common divisor.

There are various ways to generalize \eqref{eq:cesaro_seminal_result} and some of them have received attention in probabilistic as well as number theoretic literature. For example, one can ask about the asymptotic distribution (instead of the asymptotic average) of sequences of random variables
$$
\left(\frac{\LCM(U_1^{(n)},U_2^{(n)},\ldots,U_{m}^{(n)})}{U_1^{(n)}U_2^{(n)}\cdots U_m^{(n)}}\right)_{n\in\N}\quad\text{or}\quad \left(\frac{\LCM(U_1^{(n)},U_2^{(n)},\ldots,U_{m}^{(n)})}{n^m}\right)_{n\in\N}
$$
where $m\in\N$ is a fixed integer. This problem has been solved in \cite{BosMarRas:19}, see also \cite{Fernandez+Fernandez:2013,Hilberdink+Toth:2016}, by showing that both fractions converge in distribution to proper random variables, see Eq.~(15) and (16) in \cite{BosMarRas:19} and also Proposition \ref{prop:fixed_m} below. Another way towards generalization of \eqref{eq:cesaro_seminal_result} is to replace the random set $\{U_1^{(n)},U_2^{(n)},\ldots,U_{m}^{(n)}\}$ by a more sophisticated (and/or with cardinality depending on $n$) random subset of $[n]$. In \cite{CillRueSarka:14} the following model has been proposed: fix $\theta\in(0,1)$, remove every element $j\in [n]$, independently of the other elements, with probability $1-\theta$ and denote the remaining subset by $A_n(\theta)$. This model  has been intensively analyzed in \cite{AlsKabMar:19}, see also \cite{Sanna:2020}, where the authors proved various limit theorems, including a strong law of large numbers for $L_n(\theta):=\log \LCM(A_n(\theta))$, functional limit theorems  for the process $(L_{\lfloor nt\rfloor}(\theta))_{t\in [0,\,1]}$ and Poisson limit theorems when $\theta=\theta_n$ is either close to zero or one.

Another class of examples is related to the theory of random permutations. Let $\mathfrak{S}_n$ be the symmetric group of all permutations of $[n]$. If $\sigma\in\mathfrak{S}_n$ is a random permutation picked according to some probability distribution on $\mathfrak{S}_n$, then the
 collection of cycle lengths of $\sigma$ is a random subset of $[n]$ and its least common multiple is equal to the order ${\rm ord}(\sigma)$ of the permutation $\sigma$. This object has been studied for Ewens' permutations (including the case of the uniform distribution on $\mathfrak{S}_n$), see \cite{Arratia+Tavare:1992}, and also for more general probability measures on $\mathfrak{S}_n$, see \cite{Gnedin+Iksanov+Marynych:2012,Storm+Zeindler:2015}. For example, the famous Erd\H{o}s--Tur\'{a}n law \cite{Erdos+Turan:1967} states that $\log {\rm ord}(\sigma)$ is asymptotically normal, as $n\to\infty$, if $\sigma$ is picked uniformly at random from $\mathfrak{S}_n$.

In the present paper we investigate another model of choosing a random subset of $[n]$. Let $B_n(m)$ be a subset chosen uniformly at random among all subsets of $[n]$ containing exactly $m$ elements. To the best of our knowledge, the asymptotic behavior of 
$\log\LCM(B_n(m))$ has only been (partly) investigated in \cite{CillRueSarka:14}. Theorem 1.2 therein states that for $m=m_n<n$,
$$
\frac{(1-m/n)\log \LCM(B_n(m))}{m\log(n/m)}\toprobab 1,
$$
where $\toprobab$ denotes convergence in probability. The purpose of our work
is to shed more light on this model. We provide here an explicit pathwise construction of the whole random collection $(B_n(m))_{1\leq m\leq n}$ and then prove new limit theorems for the properly normalized sequence of stochastic processes
$$
\Big(\log \LCM(B_n(\lfloor t m\rfloor ))\Big)_{t\geq 0}
$$
assuming that $m = o(n)$ and $m,n\to\infty$ (Theorem \ref{thm:main_result1}). As an important ingredient of our approach we also derive limit theorems for accompanying random processes
$$
\left(\log \LCM({U_1^{(n)},\ldots,U_{\lfloor m t\rfloor }^{(n)}})\right)_{t\geq0},
$$
where $(U_j^{(n)})_{j\geq 1}$ is a sequence of independent copies of a
random variable $U^{(n)}$ with distribution \eqref{eq:uniform_distribution}, and both $m,n\to\infty$ (Theorem \ref{thm:main_result2}).
The case of fixed $m$ follows essentially from the aforementioned results proved in \cite{BosMarRas:19}, however the case when $m_n\to \infty$ requires a completely different and novel approach, whilst existing techniques fail for this problem.

The paper has the following structure. In Section \ref{sec:coupon_collector} we provide an explicit construction of the collection
$(B_n(m))_{1\leq m\leq n}$ and relate it to the classical coupon collector problem. Section \ref{sec:coupon_collector} culminates with formula \eqref{eq:redefinition}, which defines all $B_n(m)$, $1\leq m\leq n$ on the same probability space via the accompanying sequence $(U_j^{(n)})_{j\geq 1}$ and an appropriate sequence of stopping times $(\tau^{(n)}(m))_{m\geq 1}$. The main results are presented in Subsections \ref{subsec:fixed_m} (Proposition \ref{prop:fixed_m} and Theorem \ref{thm:fixed_m}) and \ref{subsec:diverging_m} (Theorems \ref{thm:main_result1} and \ref{thm:main_result2}). Our strategy of the proof of the main results is discussed in Subsection \ref{subsec:strategy}. In Section \ref{sec:coupon_aux} some elementary facts  about the sequence $(\tau^{(n)}(m))_{m\geq 1}$, used later in the proofs, are collected. The proofs of Theorems \ref{thm:main_result1} and \ref{thm:main_result2} are scattered across Sections \ref{sec:central_moments}, \ref{sec:proof_1} and \ref{sec:proof_2}. The proof of Proposition \ref{prop:fixed_m} is given in the Appendix.

\section{Explicit construction of \texorpdfstring{$(B_n(m))_{1\leq m\leq n}$}{Bn(m)}}\label{sec:coupon_collector}
Let $(U_j^{(n)})_{j\in\N}$ be a sequence of independent copies of a random variable $U^{(n)}$ with the uniform distribution on $[n]$, that is, the distribution given by~\eqref{eq:uniform_distribution}. For $n\in\N$, $1\leq m\leq n$, put
\begin{multline*}
\tau^{(n)}(m):=\inf\{j\in\N:\text{there are exactly }m\text{ different values among } \\U_1^{(n)},U_2^{(n)},\ldots,U_j^{(n)}\}.
\end{multline*}
Note that $\P\{m\leq \tau^{(n)}(m)<\infty\}=1$. The variable $\tau^{(n)}(m)$ can be interpreted in terms of the well-known coupon collector problem. Assume that there are infinitely many coupons and that each coupon has one of $n$ different types. Then $\tau^{(n)}(m)$ is the number of coupons a collector needs to buy in order to have coupons of exactly $m$ different types. In particular, $\tau^{(n)}(n)$ is the number of purchases needed to acquire the entire collection. We put $\tau^{(n)}(0):=0$ and $\tau^{(n)}(m):=\tau^{(n)}(n)$ for $m>n$.

The following representation is immediate. For $i=1,\ldots,n$, let $X_{i,n}$ be the number of purchases needed to buy a coupon of a new type given that the number of different types in the current collection is $i-1$. It is straightforward to check that the variables $(X_{i,n})_{i=1,\ldots,n}$ are mutually independent and
$$
\P\{X_{i,n}=j\}=\left(\frac{i-1}{n}\right)^{j-1}\left(1-\frac{i-1}{n}\right),\quad j\in\N.
$$
Furthermore,
\begin{equation}\label{eq:tau_sum_representation}
\tau^{(n)}(m)=X_{1,n}+X_{2,n}+\cdots+X_{m,n},\quad 1\leq m\leq n.
\end{equation}

We have the following result. 

\begin{lemma}\label{lem:explicit_construction}
For every fixed $1\leq m\leq n$, the random set
$$
\{U_{1}^{(n)},U_{2}^{(n)},\ldots,U_{\tau^{(n)}(m)}^{(n)}\}
=\{U_{\tau^{(n)}(1)}^{(n)},U_{\tau^{(n)}(2)}^{(n)},\ldots,U_{\tau^{(n)}(m)}^{(n)}\}
$$
is uniformly distributed among all subsets of $[n]$ containing exactly $m$ elements.
\end{lemma}
\begin{proof}
It is enough to show that
$$
\P\{\{U_{\tau^{(n)}(1)}^{(n)},U_{\tau^{(n)}(2)}^{(n)},\ldots,U_{\tau^{(n)}(m)}^{(n)}\}=\{b_1,b_2,\ldots,b_m\}\}=1/\binom{n}{m},\quad 1\leq m\leq n.
$$
for every fixed set $\{b_1,b_2,\ldots,b_m\}\subset [n]$ of pairwise distinct integers. We argue by induction on $m$. For $m=1$, the claim is obvious because $\tau^{(n)}(1)=1$ by definition. We have, for $2\leq m\leq n$,
\begin{align*}
&\hspace{-0.2cm}\P\{\{U_{\tau^{(n)}(1)}^{(n)},U_{\tau^{(n)}(2)}^{(n)},\ldots,U_{\tau^{(n)}(m)}^{(n)}\}=\{b_1,b_2,\ldots,b_m\}\}\\
&\sum_{j=1}^{m}\P\{U_{\tau^{(n)}(m)}^{(n)}=b_j|U_{\tau^{(n)}(1)}^{(n)},U_{\tau^{(n)}(2)}^{(n)},\ldots,U_{\tau^{(n)}(m-1)}^{(n)}=\{b_1,b_2,\ldots,b_m\}\setminus \{b_j\}\}\\
&\times\P\{U_{\tau^{(n)}(1)}^{(n)},U_{\tau^{(n)}(2)}^{(n)},\ldots,U_{\tau^{(n)}(m-1)}^{(n)}=\{b_1,b_2,\ldots,b_m\}\setminus \{b_j\}\}.
\end{align*}
The first probability under the sum is equal to $(n-m+1)^{-1}$ and the second probability, by the induction assumption, is $\binom{n}{m-1}^{-1}$. Therefore,
\begin{multline*}
\P\{\{U_{\tau^{(n)}(1)}^{(n)},U_{\tau^{(n)}(2)}^{(n)},\ldots,U_{\tau^{(n)}(m)}^{(n)}\}=\{b_1,b_2,\ldots,b_m\}\}\\
=\sum_{j=1}^{m}\Big((n-m+1)\binom{n}{m-1}\Big)^{-1}=m \Big((n-m+1)\binom{n}{m-1}\Big)^{-1}=1/\binom{n}{m},
\end{multline*}
as wanted.
\end{proof}
The explicit construction in Lemma \ref{lem:explicit_construction} allows us to construct the whole collection $(B_n(m))_{1\leq m\leq n}$ in a consistent way from the sequence  $(U_j^{(n)})_{j\in\N}$. Thus, from now on we {\it redefine} the sets $B_n(m)$ by putting
\begin{equation}\label{eq:redefinition}
B_n(m):=\{U_{1}^{(n)},U_{2}^{(n)},\ldots,U_{\tau^{(n)}(m)}^{(n)}\},\quad 1\leq m\leq n.
\end{equation}
Now we can speak about the distribution of $(B_n(i),B_n(j))$, $i\neq j$ and more generally about finite-dimensional distributions of the collection $(B_n(m))_{1\leq m\leq n}$.

\section{Limit theorems for the least common multiple of \texorpdfstring{$B_n(m)$}{Bn(m)}}

Denote by $\mathcal{P}$ the set of prime numbers. Also, let $\lambda_p(n)$ denote the multiplicity of a prime number $p\in\mathcal{P}$
in the unique decomposition of $n\in\N$ into prime factors, that is,
$$ n=\prod_{p\in\mathcal{P}}p^{\lambda_p(n)}.$$
In what follows we tacitly assume that all products and sums with indices $p,q,r,s$ only extend over prime numbers.
We also stipulate that ``${\rm const}$'' is a constant whose value is of no importance and may change from one appearance to another. Also, all unspecified limit relations are assumed to hold as $n\to\infty$.

\subsection{The case of fixed \texorpdfstring{$m$}{m}.}\label{subsec:fixed_m}
Let $((\mathcal{G}_k(2),\mathcal{G}_k(3),\mathcal{G}_k(5),\ldots))_{k\in\N}$ be a sequence of independent copies of an infinite vector $(\mathcal{G}(2),\mathcal{G}(3),\mathcal{G}(5),\ldots)$ with mutually independent coordinates having a geometric distribution
\begin{equation}\label{geom}
\P\{\mathcal{G}(p)\geq j\}=p^{-j},\quad j\in\N_0,\quad p\in\mathcal{P}.
\end{equation}
The importance of these geometric variables stems from the following lemma which has a long history, see, for instance, \cite[Formulas (2.5)-(2.7)]{Kubilius:1964} and \cite{Bill:74}, and is presented here in the form borrowed from \cite{BosMarRas:19}.

The distribution given in \eqref{eq:uniform_distribution} is a discrete uniform distribution. We recall that there also exists a continuous uniform distribution $\mu$, say on $[0,1]$ defined by $\mu({\rm d}x)=\1_{(0,1)}(x){\rm d}x$. We shall write $\todistri$ and  $\todistr$ to denote convergence in distribution in $\R^{\infty}$ endowed with the product topology and in $\mr$, respectively.
\begin{lemma}\label{lem:conv_to_geom}
Let
\begin{equation*}
     U^{(n)}=\prod_{p\in\mathcal{P}}p^{\lambda_p(U^{(n)})}
\end{equation*}
be the decomposition of $U^{(n)}$ with distribution \eqref{eq:uniform_distribution} into prime factors. Then
\begin{enumerate}[label={(\roman{*})},ref={(\roman{*})}]
     \item\label{itt:1}
\begin{equation*}
     \bigl(\lambda_p(U^{(n)})\bigr)_{p\in\mathcal{P}}\todistri \Big(\mathcal{G}(p)\Big)_{p\in\mathcal{P}};
\end{equation*}

     \item\label{itt:2}
\begin{equation*}
     \left(n^{-1} U^{(n)},\bigl(\lambda_p(U^{(n)})\bigr)_{p\in\mathcal{P}}\right)\todistri \left(U,\left(\mathcal{G}(p)\right)_{p\in\mathcal{P}}\right),
\end{equation*}
with $U$ being uniformly distributed on $[0,1]$ and independent of $\left(\mathcal{G}(p)\right)_{p\in\mathcal{P}}$;
     \item\label{itt:3}for $p,q\in\mathcal{P}$, $p\neq q$ and $k_p,k_q\in\N_0$, 
\begin{equation*}
     \P\{\lambda_p(U^{(n)})=k_p,\lambda_q(U^{(n)})=k_q\}=(1-p^{-1})(1-q^{-1})p^{-k_p}q^{-k_q}+O(n^{-1}),
\end{equation*}
where the constant in the $O$-term does not depend on $(p,q,k_p,k_q)$.
\end{enumerate}
\end{lemma}

With the help of this lemma the following result has been proved in \cite{BosMarRas:19}. See also \cite{Fernandez+Fernandez:2013} for the cases $m=2,3$.
\begin{proposition}[Formula (16) in \cite{BosMarRas:19}]\label{thm:fixed_m_BMR:19}
For every fixed $m\in\N$,
\begin{multline*}
\log \LCM(U_1^{(n)},U_2^{(n)},\ldots,U_m^{(n)})-m\log n\\
\todistr \sum_{j=1}^m \log U_j+\sum_p \log p\cdot \Big(\max_{1\leq k\leq m}\mathcal{G}_k(p)-\sum_{k=1}^{m}\mathcal{G}_k(p)\Big),
\end{multline*}
where $(U_j)_{j=1,\ldots,m}$ are independent random variables with the uniform distribution on $[0,1]$ which are also independent of $(\mathcal{G}_k(p))_{k\in\N,p\in\mathcal{P}}$.
\end{proposition}

Using the same techniques as in \cite{BosMarRas:19} Proposition \ref{thm:fixed_m_BMR:19} can be strengthened as follows.
\begin{proposition}\label{prop:fixed_m}
\begin{multline*}
\Big(\log \LCM(U_1^{(n)},U_2^{(n)},\ldots,U_m^{(n)})-m\log n\Big)_{m\in\N}\\
\todistri \Big(\sum_{j=1}^m \log U_j+\sum_p \log p\cdot \Big(\max_{1\leq k\leq m}\mathcal{G}_k(p)-\sum_{k=1}^m \mathcal{G}_k(p)\Big)\Big)_{m\in\N},
\end{multline*}
where $(U_j)_{j\geq 1}$ are independent random variables  with the uniform distribution on $[0,1]$ which are also independent of $(\mathcal{G}_k(p))_{k\in\N,p\in\mathcal{P}}$.
\end{proposition}

We shall give a short proof of Proposition \ref{prop:fixed_m} in the Appendix. Since, for every fixed $m$,
$$
\lim_{n\to\infty}\P\{\tau^{(n)}(m)=m\}=1,
$$
see formula \eqref{eq:tau_to_m_convergence} below, Proposition \ref{prop:fixed_m} immediately yields the following.
\begin{theorem}\label{thm:fixed_m}
\begin{multline*}
\Big(\log \LCM(B_n(m))-m\log n\Big)_{m\in\N}\\
\todistri \Big(\sum_{j=1}^m\log U_j+\sum_p \log p\cdot \Big(\max_{1\leq k\leq m}\mathcal{G}_k(p)-\sum_{k=1}^{m}\mathcal{G}_k(p)\Big)\Big)_{m\in\N}.
\end{multline*}
\end{theorem}

\subsection{The case \texorpdfstring{$m=m_n\to\infty$}{m->infinity} and \texorpdfstring{$m_n=o(n)$}{m=o(n)}.}\label{subsec:diverging_m}

Theorem \ref{thm:fixed_m} dealing with the case of fixed $m$ follows, for the most part, from the previously known results. The case $m_n\to\infty$ turns out to be more intriguing and requires a different approach.

As usual, we write $(\mathcal{Z}_n(t))_{t\geq 0}\tofd (\mathcal{Z}(t))_{t\geq 0}$ to  denote weak  convergence  of  finite-dimensional distributions, that is, for any $k\in\mn$ and any $0\leq t_1<t_2<\cdots<t_k<\infty$, $(\mathcal{Z}_n(t_1),\ldots,\mathcal{Z}_n(t_k))$ converges in distribution to $(\mathcal{Z}(t_1),\ldots, \mathcal{Z}(t_k))$ as $n\to\infty$. For every fixed $n\in\N$ and $y\geq 0$, put
\begin{equation}\label{eq:cn}
c_n(y):=\sum_{p\leq n}\log p\, (1-(1-n^{-1}\lfloor n/p\rfloor)^y).
\end{equation}
We distinguish two cases:
\begin{itemize}
\item[(A)] $m_n\leq n^{1/2}$ for all sufficiently large $n$ and $\lim_{n\to\infty}m_n=\infty$;
\item[(B)] $m_n>n^{1/2}$ for all sufficiently large $n$ and $m_n=o(n)$ as $n\to\infty$.
\end{itemize}
Here is our first main result.
\begin{theorem}\label{thm:main_result1}
Let $(B(t))_{t\geq 0}$ be a standard Brownian motion.
\begin{itemize}
\item[(i)] If (A) holds, then
$$
\left(\frac{\log\LCM(B_n(\lfloor m_nt\rfloor))-c_n(\lfloor m_n t\rfloor)}{\sqrt{~2^{-1}m_n}\log m_n}\right)_{t\geq 0}
\tofd (B(t))_{t\geq 0}.
$$
\item[(ii)] If (B) holds and $m_n=O(n(\log n)^{-1})$, then
$$
\left(\frac{\log\LCM(B_n(\lfloor m_n t\rfloor))-c_n( -n\log (1-(m_n t)/n) )}{\sqrt{2^{-1}m_n(\log n-\log m_n)(3\log m_n-\log n)}}\right)_{t\geq 0}
\tofd (B(t))_{t\geq 0}.
$$
\end{itemize}
\end{theorem}
Put
$$
Y_n(m):=\log \LCM (U_1^{(n)},U_2^{(n)},\ldots,U^{(n)}_{m}),\quad 1\leq m\leq n
$$
and note that
\begin{multline}\label{eq:random_change_of_time}
\log\LCM(B_n(\lfloor m_n t\rfloor))= \log \LCM (U_1^{(n)},U_2^{(n)},\ldots,U^{(n)}_{\tau^{(n)}(\lfloor m_n t\rfloor)})\\
=Y_n(\tau^{(n)}(\lfloor m_n t\rfloor)),\quad t\geq 0.
\end{multline}
We deduce Theorem \ref{thm:main_result1} from the following result, a counterpart of Proposition \ref{prop:fixed_m} for diverging $m_n$, which is interesting on its own.
\begin{theorem}\label{thm:main_result2}
Let $(B(t))_{t\geq 0}$ be a standard Brownian motion.
\begin{itemize}
\item[(i)] If (A) holds, then
$$
\left(\frac{Y_n(\lfloor m_n t\rfloor)-c_n(\lfloor m_n t\rfloor)}{\sqrt{~2^{-1}m_n}\log m_n}\right)_{t\geq 0}
\tofd (B(t))_{t\geq 0}.
$$
\item[(ii)] If (B) holds, then
$$
\left(\frac{Y_n(\lfloor m_n t\rfloor)-c_n(\lfloor m_n t\rfloor)}{\sqrt{2^{-1}m_n(\log n-\log m_n)(3\log m_n-\log n)}}\right)_{t\geq 0}
\tofd (B(t))_{t\geq 0}.$$
\end{itemize}

\end{theorem}

\subsection{Strategy of proof of Theorem \ref{thm:main_result2}}\label{subsec:strategy}

For some samples $\{c^{(n)}_1,\ldots, c^{(n)}_{m_n}\}$ of random integers taking values in $[n]$ the logarithm of the least common multiple
\begin{equation}\label{lcm}
\log \LCM(\{c_1^{(n)},\ldots, c^{(n)}_{m_n}\})=\sum_{p\leq n}\log p\,\max_{1\leq k\leq m_n}\,\lambda_p(c^{(n)}_k)
\end{equation}
may be well-approximated by $\log \prod_{k=1}^{m_n}c_k^{(n)}$. For instance, this is known to be the case when $\{c_1^{(n)},\ldots, c_{m_n}^{(n)}\}$ are the cycle lengths (with $m_n$ being the total number of cycles) of a wide class of random permutations including Ewens' permutations, see \cite{Gnedin+Iksanov+Marynych:2012,Storm+Zeindler:2015}. Intuitively, such an approximation is successful provided that `most' of the values among $c^{(n)}_1,\ldots, c^{(n)}_{m_n}$ are distinct and `most' of the positive multiplicities $\lambda_p(c_k^{(n)})$, $p\in\mathcal{P}$, $k=1,\ldots, m_n$ are ones. Of course, many samples do not enjoy these properties and particularly neither do $\{U_1^{(n)},\ldots, U_{m_n}^{(n)}\}$ that we are focused on.

Roughly speaking, the previous approximation argument is based on comparison of $\max_{1\leq k\leq m_n}\,\lambda_p(c^{(n)}_k)$ and $\sum_{k=1}^{m_n}\lambda_p(c^{(n)}_k)$. However, it seems that in many cases $\max_{1\leq k\leq m_n}\,\lambda_p(c^{(n)}_k)$ should be closer to $\sum_{k=1}^{m_n}\1_{\{\lambda_p(c^{(n)}_k)\geq 1\}}$ rather than to $\sum_{k=1}^{m_n}\lambda_p(c^{(n)}_k)$, and our strategy is to exploit this line of reasoning. We shall show in Lemma \ref{lem:red1} that $Y_n(\lfloor m_n t\rfloor)=\log \LCM(U_1^{(n)},\ldots, U^{(n)}_{\lfloor m_n t\rfloor})$ is well-approximated by $\sum_{p\leq n}\log p\, \1_{\{\max_{1\leq k\leq \lfloor m_n t\rfloor}\lambda_p(U_k^{(n)})\geq 1\}}$. Furthermore, we shall prove in Lemmas \ref{red3} and \ref{red4}, respectively, that `small' primes $p\leq m_n$ do not give significant contribution to $Y_n(\lfloor m_n t\rfloor)$ and that in the range of `large' primes $p>m_n$ the indicators $\1_{\{\max_{1\leq k\leq \lfloor m_n t\rfloor}\lambda_p(U_k^{(n)})\geq 1\}}$ can be safely replaced by $\sum_{k=1}^{\lfloor m_n t\rfloor}\1_{\{\lambda_p(U_k^{(n)})\geq 1\}}$. Summarizing, we are going to approximate $Y_n(\lfloor m_n t\rfloor)$ by
$$
\sum_{m_n<p\leq n}\log p\sum_{k=1}^{\lfloor m_n t\rfloor}\1_{\{\lambda_p(U_k^{(n)})\geq 1\}}=\sum_{k=1}^{\lfloor m_n t\rfloor}\sum_{m_n<p\leq n}\log p\cdot \1_{\{\lambda_p(U_k^{(n)})\geq 1\}}
$$ which is the sum of independent random variables. A limit theorem for the approximating processes is given in Lemma \ref{lem:limit}.

We close this section by a discussion of inefficiency of a tempting alternative approach. A specialization of formula \eqref{lcm} reads 
$$
Y_n(\lfloor m_n t\rfloor)=\sum_{p\leq n}\log p \max_{1\leq k\leq \lfloor m_n t\rfloor}\lambda_p(U_k^{(n)}),\quad t\geq 0.
$$
As far as a proof of Theorem \ref{thm:main_result2} is concerned, a naive idea justified in part by  Lemma \ref{lem:conv_to_geom} is to replace the terms $\lambda_p(U_k^{(n)})$ with their limits $\mathcal{G}_k(p)$ and to approximate $(Y_n(\lfloor m_n t\rfloor))_{t\geq 0}$ by $(\widehat Y_n(\lfloor m_n t\rfloor))_{t\geq 0}$, where
\begin{equation}\label{eq:y_n_geom_def}
\widehat{Y}_n(\lfloor m_n t\rfloor)=\sum_{p\leq n}\log p \max_{1\leq k\leq \lfloor m_n t\rfloor}\mathcal{G}_k(p),\quad t\geq 0.
\end{equation}
The right-hand side is the sum of independent random variables which ensures that the analysis of $\widehat{Y}_n(\lfloor m_n t\rfloor)$ is simple. 
It turns out that unless $\log n/\log m_n\to 1$, that is, $m_n$ is rather close to $n$, this intuition is {\it wrong} in a sense that the limit relation in Theorem \ref{thm:main_result2} does not hold with $Y_n$ replaced by $\widehat{Y}_n$. The details can be found in Proposition \ref{prop:geom} given in the Appendix. This is an unexpected and peculiar phenomenon because most of the known results in probabilistic number theory involving discrete
uniform random variables can be proved using the formal substitution $\lambda_p(U^{(n)})\mapsto \mathcal{G}(p)$. The list includes:
\begin{itemize}
\item[(i)] the Hardy--Ramanujan central limit theorem for the number of prime divisors;
\item[(ii)] the Erd\H{o}s--Kac central limit theorem for strongly additive functions \cite{Erdos+Kac:1940};
\item[(iii)] the functional central limit theorem for the counts of prime factors \cite{Bill:74};
\item[(iv)] the Kubilius theorems on convergence to infinitely divisible laws \cite[Section 5]{Kubilius:1964};
\item[(v)] Proposition \ref{prop:fixed_m} of the present paper.
\end{itemize}
A detailed discussion, proofs and further examples can be found in \cite{Arratia:2002} and in Section 1.2 of \cite{Arratia+Barbour+Tavare:2003}. In particular, an optimal coupling between $(\lambda_p(U^{(n)})_{p\in\mathcal{P}}$ and $(\mathcal{G}(p))_{p\in\mathcal{P}}$ is constructed in \cite{Arratia:2002}. A partial explanation of the inefficiency of this approach in our situation is that the cumulative error caused by replacing $\lambda_p(U_k^{(n)})$ by $\mathcal{G}_k(p)$, $k=1,\ldots,m_n$, which is negligible when the number $m_n$ of such replacements is bounded (as in examples (i)-(v) above), becomes significant with a growth of the sample.

\section{Some auxiliary results related to the coupon collector problem}\label{sec:coupon_aux}

In this section we discuss the asymptotic behaviour of the stopping time $\tau^{(n)}(m_n)$ as $n\to\infty$. According to formula \eqref{eq:random_change_of_time}, this information is of principal importance for deducing  Theorem \ref{thm:main_result1} from Theorem \ref{thm:main_result2}. Using \eqref{eq:tau_sum_representation} we infer
\begin{equation}\label{eq:tau_expectation}
\E \tau^{(n)}(m)=\sum_{k=1}^{m}\E X_{k,n}=\sum_{i=1}^{m}n(n-i+1)^{-1}=n(H_n-H_{n-m}),
\end{equation}
where $H_n:=\sum_{k=1}^{n}k^{-1}$ is the $n$th harmonic number. Moreover,
\begin{multline*}
\Var \tau^{(n)}(m)=\sum_{k=1}^{m}\Var X_{k,n}=\sum_{k=1}^m ((k-1)/n)((n-k+1)/n)^{-2}\\
=n\sum_{k=n-m+1}^n k^{-2}(n-k)=n^2(H_{n,2}-H_{n-m,2})-n(H_n-H_{n-m}),
\end{multline*}
where $H_{n,2}:=\sum_{k=1}^n k^{-2}$. 

Assume now that $m=m_n<n$ depends on $n$ is such a way that $m_n=o(n)$. By using the standard expansions
$$H_n=\log n+\gamma+(2n)^{-1}+O(n^{-2}),\quad H_{n,2}=\zeta(2)-n^{-1}+2^{-1}n^{-2}+O(n^{-3}),$$ where $\gamma$ is the Euler-Mascheroni constant,
we obtain
\begin{equation}\label{eq:tau_expansion0}
\E \tau^{(n)}(m_n)=m_n+O\left(m_n^2/n\right),\quad \Var \tau^{(n)}(m_n)=O(m_n^2/n).
\end{equation}
In particular, if $m_n=o(n^{1/2})$, then
$$
\tau^{(n)}(m_n)-m_n\toprobab 0,
$$
or, in other words,
\begin{equation}\label{eq:tau_to_m_convergence}
\lim_{n\to\infty}\P\{\tau^{(n)}(m_n)=m_n\}=1.
\end{equation}

\section{Asymptotics of the central moments \texorpdfstring{of $\sum_{m_n<p\leq n}\log p\cdot \1_{\{\lambda_p(U_k^{(n)})\geq 1\}}$}{}}\label{sec:central_moments}

For $n\in\N$ and $1\leq m<n$, put
$$
\widetilde{U}^{(n,m)}=\prod_{m<p\leq n} p^{\1_{\{\lambda_p(U^{(n)})\geq 1\}}}.$$ In this section we investigate the
behavior of $\E \Big(\log \widetilde{U}^{(n,m_n)}-\E \log \widetilde{U}^{(n,m_n)}\Big)^{2s}$ for $s=1,2$ as $n\to\infty$. The results obtained here
are an important ingredient in the proof of Theorem \ref{thm:main_result2}.

\subsection{Auxiliary results}
We start with several auxiliary facts.
\begin{lemma}\label{lem:log_uniform_moments}
For $s\in\N$,
$$
\lim_{n\to\infty}\E \left(\log U^{(n)}-\E \log U^{(n)}\right)^{2s}=\E (\mathcal{E}_1-\E\mathcal{E}_1)^{2s}=\E (\mathcal{E}_1-1)^{2s},
$$
where $\mathcal{E}_1$ is a random variable with the exponential distribution of unit mean.
\end{lemma}
\begin{proof}
To justify this it is tempting to use distributional convergence of $\log n-\log U^{(n)}$ to $\mathcal{E}_1$ in combination with a uniform integrability argument. However, we find it simpler to exploit a more analytic approach based on a direct calculation of the moments: for $r\in\N$,
\begin{multline*}
\E (\log U^{(n)})^r=n^{-1}\sum_{k=1}^n \log^r k=n^{-1}\int_1^n\log^r x\,{\rm d}x+O(n^{-1}\log^r n)\\=\sum_{i=0}^r (-1)^i i! \binom{r}{i} \log^{r-i} n+O(n^{-1}\log^r n). 
\end{multline*}
This follows from the Euler--Maclaurin summation formula. An application of the binomial theorem completes the proof.
\end{proof}

As usual, $a_n\sim b_n$ as $n\to\infty$ means that $\lim_{n\to\infty}(a_n/b_n)=1$.

\begin{lemma}\label{lem:inter}
Let $s\in\N$ and $(X_n)_{n\geq 1}$ and $(Y_n)_{n\geq 1}$ be arbitrarily dependent sequences of random variables with finite moments of order $2s$. If $\E X_n^{2s}=O(1)$ and $\lim_{n\to\infty}\E Y_n^{2s}=\infty$, then $\E (X_n-Y_n)^{2s}\sim \E Y_n^{2s}$ as $n\to\infty$.
\end{lemma}
\begin{proof}
We start with a representation $$\E (X_n-Y_n)^{2s}=\sum_{k=0}^{2s}(-1)^k \binom{2s}{k}\E X_n^{2s-k}Y_n^k.$$ For $k=1,\ldots, 2s-1$, an application of H\"{o}lder's inequality yields $$|\E X_n^{2s-k}Y_n^k|\leq (\E X_n^{2s})^{1-k/(2s)} (\E Y_n^{2s})^{k/(2s)}=o(\E Y_n^{2s})$$ because $k<2s$. Thus, we have proved that  $\E (X_n-Y_n)^{2s}=\E Y_n^{2s}+o(\E Y_n^{2s})$.
\end{proof}

\begin{lemma}\label{lem:u_to_tilde_u}
The following asymptotic relations hold:
\begin{equation}\label{eq:k_est1}
K_n:=\E \left(\log U^{(n)}-\log \widetilde{U}^{(n)}\right)^4=O(1),\quad n\to\infty,
\end{equation}
where $\widetilde{U}^{(n)}=\prod_{p\leq n} p^{\1_{\{\lambda_p(U^{(n)})\geq 1\}}}$ for $n\in\mn$,
and
\begin{equation}\label{eq:k_est}
\E \left(\log U^{(n)}-\log \widetilde{U}^{(n)}-\E \log U^{(n)}+\E \log \widetilde{U}^{(n)}\right)^4=O(1),\quad n\to\infty.
\end{equation}
\end{lemma}
\begin{proof}
We write with the help of the binomial theorem
\begin{align*}
&\hspace{-0.1cm}K_n= \E \Big(\sum_{p\leq n}\log p\cdot (\lambda_p(U^{(n)})-\1_{\{\lambda_p(U^{(n)})\geq 1\}})\Big)^4\\
&=\sum_{p\leq n}\log^4 p\times\E (\lambda_p(U^{(n)})-1)^4_+\\
&+4\sum_{p\neq q \leq n}\log^3 p\log q\times \E (\lambda_p(U^{(n)})-1)^3_+(\lambda_q(U^{(n)})-1)_+\\
&+ 6\sum_{p<q\leq n}\log^2 p\log^2 q\times \E (\lambda_p(U^{(n)})-1)^2_+(\lambda_q(U^{(n)})-1)^2_+\\
&+12\sum_{p\neq q\neq r \leq n}\log^2 p\log q\log r\times\E (\lambda_p(U^{(n)})-1)^2_+(\lambda_q(U^{(n)})-1)_+(\lambda_r(U^{(n)})-1)_+\\ &+24\sum_{p<q<r<s\leq n} \log p\log q\log r\log s\times\\
&\hspace{0.5cm} \E (\lambda_p(U^{(n)})-1)_+(\lambda_q(U^{(n)})-1)_+(\lambda_r(U^{(n)})-1)_+ (\lambda_s(U^{(n)})-1)_+\\
&=\sum_{i=1}^5 K_i(n).
\end{align*}
For any $k\in\N$, any positive integers $v_1,\ldots, v_k$ and any distinct prime numbers $p_1,\ldots, p_k$, we have
\begin{align*}
&\hspace{-0.3cm}\E \left(\prod_{i=1}^k  (\lambda_{p_i}(U^{(n)})-1)^{v_i}_+\right)\\
&=\sum_{j_1,\ldots, j_k\geq 2}(j_1-1)^{v_1}\cdots (j_k-1)^{v_k}\P\{\lambda_{p_1}(U^{(n)})=j_1,\ldots, \lambda_{p_k}(U^{(n)})=j_k\}\\
&\leq \sum_{j_1,\ldots, j_k\geq 2}(j_1-1)^{v_1}\cdots(j_k-1)^{v_k}\P\{\lambda_{p_1}(U^{(n)})\geq j_1,\ldots, \lambda_{p_k}(U^{(n)})\geq j_k\}\\
&=\sum_{j_1,\ldots, j_k\geq 2}(j_1-1)^{v_1}\cdots (j_k-1)^{v_k}\lfloor n/(p_1^{j_1}\cdots p_k^{j_k})\rfloor n^{-1}\\
&\leq \sum_{j_1,\ldots, j_k\geq 2}j_1^{v_1}p_1^{-j_1}\cdots j_k^{v_k}p_k^{-j_k}\leq {\rm const}\cdot\prod_{i=1}^k p_i^{-2}.
\end{align*}
Now, for $i=1,\ldots,5$, the asymptotic estimate $K_i(n)=O(1)$ follows from the preceding bound, $\sum_p p^{-2}\log^4 p<\infty$ and the following inequalities:
\begin{align*}
K_2(n)&\leq {\rm const}\cdot\sum_{p\neq q \leq n}(p^{-2}\log^3 p)(q^{-2}\log q)\\
&\hspace{5cm}\leq{\rm const}\cdot\left(\sum_{p}p^{-2}\log^3 p\right)\left(\sum_{p}p^{-2}\log p\right);\\
K_3(n)&\leq {\rm const}\cdot\sum_{p<q \leq n}(p^{-2}\log^2 p)(q^{-2}\log^2 q) \leq {\rm const}\cdot\Big(\sum_{p}p^{-2}\log^2 p\Big)^2;
\end{align*}
\begin{align*}
K_4(n)&\leq {\rm const}\cdot\sum_{p\neq q\neq r \leq n}(p^{-2}\log^2 p)(q^{-2}\log q)(r^{-2}\log r)\\
&\hspace{5cm} \leq {\rm const}\cdot\Big(\sum_{p}p^{-2}\log^2 p\Big) \Big(\sum_{p}p^{-2}\log p\Big)^2;\\
K_5(n)&\leq {\rm const}\cdot\sum_{p<q<r<s \leq n} (p^{-2}\log p)(q^{-2}\log q)(r^{-2}\log r)(s^{-2}\log s)\\
&\hspace{5cm} \leq {\rm const}\cdot \Big(\sum_{p}p^{-2}\log p\Big)^4.
\end{align*}
Thus, \eqref{eq:k_est1} holds. Inequality \eqref{eq:k_est} follows immediately, because for any random variable $X$ with $\E X^4<\infty$ we have $\E (X-\E X)^4\leq 8\E X^4$.
\end{proof}

\subsection{Asymptotics of the variance}

The function $\pi$ defined by $\pi(x):=\sum_{p\leq x} 1$ for $x\geq 0$ is called a prime counting function. Recall that the prime number theorem (see, for instance, Theorem 6.2.1 in \cite{BGT}) states that
\begin{equation}\label{pnt}
\pi(x)~\sim x/\log x,~\quad x\to\infty.
\end{equation}
For later multiple use, we note that \eqref{pnt} in combination with integration by parts entails
\begin{equation}\label{pnt_integral}
\int_{I_x}f(y){\rm d}\pi(y)~\sim~\int_{I_x}(f(y)/\log y){\rm d}y,\quad x\to\infty 
\end{equation}
where $f(x)=x^{-a}\log^b x$ and either (a) $I_x=[2,\,x]$, $a\leq 1$, $b>0$ or (b) $I_x=(x,\infty)$, $a>1$, $b\geq 0$.

While investigating the asymptotics of $\Var (\log \widetilde{U}^{(n,m_n)})$ we treat the two cases (A) $m_n\leq n^{1/2}$ for large $n$ and $m_n\to\infty$ and (B) $m_n> n^{1/2}$ for large $n$ and $m_n=o(n)$ separately in Theorems \ref{lem:variance} and \ref{lem:variance2}. Here is a brief explanation of such a case distinction. In case (B), the variance of $\log \widetilde{U}^{(n,m_n)}$ is given by the sum of two terms which can exhibit different first order asymptotics (formula \eqref{eq:22'}). This necessitates us to provide several terms expansions for each. In the case considered the $m$'s are reasonably large and kill the influence of nonprincipal terms of these expansions (in particular, it turns out that two terms expansions suffice). In case (A), the aforementioned reasoning does not help. We state, without going into details, that
\begin{equation}\label{aux1001}
\Var (\log \widetilde{U}^{(n,m_n)})~=~2^{-1}\log^2 m_n+O(\log n \log^{-2}m_n),
\end{equation}
where, to the best of our knowledge, the big-oh term cannot be improved. Hence, formula \eqref{aux1001} as it stands only provides the correct asymptotics $\Var (\log \widetilde{U}^{(n,m_n)})\sim 2^{-1}\log^2 m_n$ for $m_n$ satisfying $m_n e^{-\log^{1/4}n}\to\infty$. For smaller $m$ formula \eqref{aux1001} is useless for us. In view of these issues we offer an alternative approach for the case (A). We show that $\Var (\log \widetilde{U}^{(n,m_n)})\sim \Var \Big(\sum_{p\leq m_n}\log p\cdot \1_{\{\lambda_p(U^{(n)})\geq 1\}}\Big)$. The variance on the right-hand side which is given by the sum of three terms is much easier to deal with, for the first order asymptotics of the terms is enough. On the other hand, this simple reasoning is not applicable in case (B) because the terms of the sum representing the variance may exhibit different first order behavior. For instance, \eqref{aux1003} still holds true in case (B),
whereas $$2n^{-1}\sum_{p<q\leq m_n}\log p\log q\, \lfloor n/(pq)\rfloor~\sim~ 2^{-1}\log^2 n-\log^2 (n/m_n).$$

\begin{theorem}\label{lem:variance}
Assume that $m_n\leq n^{1/2}$ for all $n$ large enough and $m_n\to\infty$ as $n\to\infty$. Then
$$
\Var \left(\log \widetilde{U}^{(n,m_n)}\right)~\sim~2^{-1} \log^2 m_n,\quad n\to\infty.
$$
\end{theorem}
\begin{proof}
From Lemma \ref{lem:log_uniform_moments} with $s=1$ we know that
\begin{equation}\label{eq:variance_log_uniform}
\Var \Big(\sum_{p\leq n}\log p\cdot  \lambda_p(U^{(n)})\Big)=\Var (\log U^{(n)})~\to~ 1. 
\end{equation}
Further, relation \eqref{eq:k_est} ensures that
\begin{equation}
\Var (\log U^{(n)}-\log \widetilde{U}^{(n)})=O(1) 
\end{equation}
because $\Var X\leq \sqrt{\E (X-\E X)^4}$ for any random variable $X$ with $\me X^4<\infty$. These asymptotic relations in combination with the inequality $(x+y)^2\leq 2(x^2+y^2)$, $x,y\in\mr$ entail
$$
\Var \Big(\sum_{p\leq n}\log p\cdot \1_{\{\lambda_p(U^{(n)})\geq 1\}} \Big)=\Var \left(\log \widetilde{U}^{(n)} \right)=O(1). 
$$
Further, by Lemma \ref{lem:inter} applied with $s=1$,
\begin{align*}
X_n&:=\log \widetilde{U}^{(n)}-\E \log \widetilde{U}^{(n)}\\
&=\sum_{p\leq n}\log p\cdot \1_{\{\lambda_p(U^{(n)})\geq 1\}}-\E \Big(\sum_{p\leq n}\log p\cdot \1_{\{\lambda_p(U^{(n)})\geq 1\}}\Big),
\end{align*}
and
$$
Y_n:=\sum_{p\leq m_n}\log p\cdot \1_{\{\lambda_p(U^{(n)})\geq 1\}}-\E \Big(\sum_{p\leq m_n}\log p\cdot \1_{\{\lambda_p(U^{(n)})\geq 1\}}\Big)
$$
we have
\begin{multline*}
\Var (\log \widetilde{U}^{(n,m_n)})=\Var \Big(\log \widetilde{U}^{(n)}-\sum_{p\leq m_n}\log p\cdot \1_{\{\lambda_p(U^{(n)})\geq 1\}}\Big)\\
\sim~\Var \Big(\sum_{p\leq m_n}\log p\cdot \1_{\{\lambda_p(U^{(n)})\geq 1\}} \Big)
\end{multline*}
provided the right-hand side diverges to infinity. Thus, it is enough to prove that
\begin{equation}\label{eq:variance_sum_to_m}
\Var \Big(\sum_{p\leq m_n}\log p\cdot \1_{\{\lambda_p(U^{(n)})\geq 1\}}\Big)~\sim~2^{-1} \log^2 m_n. 
\end{equation}

As a preparation, write
\begin{align}\label{aux1006}
&\hspace{-0.2cm}\Var \Big(\sum_{p\leq m_n}\log p\cdot \1_{\{\lambda_p(U^{(n)})\geq 1\}}\Big)\\
&=\E\Big(\sum_{p\leq m_n}\log p\cdot \1_{\{\lambda_p(U^{(n)})\geq 1\}}\Big)^2-\Big(\E\Big(\sum_{p\leq m_n}\log p\cdot \1_{\{\lambda_p(U^{(n)})\geq 1\}}\Big)\Big)^2\notag\\
&=n^{-1}\sum_{p\leq m_n}\log^2 p \lfloor n/p\rfloor+2n^{-1}\sum_{p<q\leq m_n}\log p\log q \lfloor n/(pq)\rfloor.\notag\\
&-\Big(n^{-1}\sum_{p\leq m_n}\log p \lfloor n/p\rfloor\Big)^2\notag
\end{align}
Here, the equalities $\mmp\{\lambda_p(U^{(n)})\geq 1\}=n^{-1}\lfloor n/p\rfloor$ and, for $p\neq q$, $$\mmp\{\lambda_p(U^{(n)})\geq 1, \lambda_q(U^{(n)})\geq 1\}=n^{-1}\lfloor n/(pq)\rfloor$$ have to be recalled.

We start by analyzing the first term
$$\sum_{p\leq m_n} p^{-1} \log^2 p-n^{-1}\sum_{p\leq m_n} \log^2 p \leq n^{-1}\sum_{p\leq m_n} \log^2 p\, \lfloor n/p \rfloor \leq
\sum_{p\leq m_n} p^{-1} \log^2 p,$$ so that
\begin{equation}\label{aux1003}
n^{-1}\sum_{p\leq m_n}\log^2 p \lfloor n/p\rfloor~\sim~ 2^{-1}\log^2 m_n 
\end{equation}
is a consequence of
\begin{multline*}
\sum_{p\leq m_n} p^{-1}\log^2 p ~\sim~ 2^{-1} \log^2 m_n\quad \text{and}\\
n^{-1}\sum_{p\leq m_n}\log^2 p~\sim~ n^{-1} m_n\log m_n=o(\log m_n). 
\end{multline*}
The latter limit relations are justified by \eqref{pnt_integral} with $f(x)=x^{-1}\log^2 x$ and $f(x)=\log^2 x$, respectively. Similarly,
\begin{equation}\label{aux1004}
\Big(n^{-1}\sum_{p\leq m_n}\log p \lfloor n/p\rfloor\Big)^2~\sim~ \log^2 m_n 
\end{equation}
follows from $$\sum_{p\leq m_n} p^{-1}\log p- n^{-1}\sum_{p\leq m_n}\log p~\leq~ n^{-1}\sum_{p\leq m_n} \log p\, \lfloor n/p \rfloor\leq \sum_{p\leq m_n}p^{-1}\log p$$ and
$$\sum_{p\leq m_n} p^{-1}\log p ~\sim~ \log m_n,\quad\quad n^{-1}\sum_{p\leq m_n}\log p ~\sim~ n^{-1} m_n=o(1).$$ Finally, we use
\begin{multline*}
\sum_{p<q\leq m_n}(p^{-1}\log p)(q^{-1}\log q)- n^{-1}\sum_{p<q\leq m_n}\log p\log q \\
\leq n^{-1}\sum_{p<q\leq m_n}\log p\log q\, \lfloor n/(pq)\rfloor \leq \sum_{p<q\leq m_n}(p^{-1}\log p)(q^{-1}\log q)
\end{multline*}
together with $$2\sum_{p<q\leq m_n}(p^{-1}\log p)(q^{-1}\log q)=\Big(\sum_{p\leq m_n}p^{-1}\log p\Big)^2-\sum_{p\leq m_n}p^{-2}\log^2 p~\sim~\log^2 m_n$$ and $$2n^{-1}\sum_{p<q\leq m_n}\log p\log q\leq n^{-1}\Big(\sum_{p\leq m_n}\log p\Big)^2~\sim~n^{-1} m_n^2=O(1)$$ (recall that $m_n\leq n^{1/2}$ by assumption)
to obtain
\begin{equation}\label{aux1005}
2n^{-1}\sum_{p<q\leq m_n}\log p\log q\, \lfloor n/(pq)\rfloor~\sim~ \log^2 m_n.
\end{equation}
A combination of \eqref{aux1003}, \eqref{aux1004} and \eqref{aux1005} proves \eqref{eq:variance_sum_to_m}.
\end{proof}

\begin{theorem}\label{lem:variance2}
Assume that $m_n>n^{1/2}$ for all $n$ large enough and $m_n=o(n)$ as $n\to\infty$. Then \begin{equation}\label{aux1009}
\Var (\log \widetilde{U}^{(n,m_n)})~\sim~2^{-1}(\log n-\log m_n)(3\log m_n-\log n),\quad n\to\infty.
\end{equation}
\end{theorem}
\begin{proof}
Note that $$\log \widetilde{U}^{(n,m_n)}-\me \log \widetilde{U}^{(n,m_n)}=\sum_{m_n< p\leq n}\log p(\1_{\{\lambda_p(U^{(n)})\geq 1\}}-\mmp\{\lambda_p(U^{(n)})\geq 1\})$$ whence
\begin{equation}\label{eq:22'}
\begin{split}
{\rm Var}(\log \widetilde{U}^{(n,m_n)})=\; &
n^{-1}\sum_{m_n<p\leq n}\log^2 p\, \lfloor n/p\rfloor-\Big(n^{-1}\sum_{m_n<p\leq n}\log p\, \lfloor n/p\rfloor\Big)^2\\&+2n^{-1}\sum_{m_n<p<q\leq n}\log p\log q\, \lfloor n/(pq)\rfloor\\=\; &
n^{-1}\sum_{m_n<p\leq n}\log^2 p\, \lfloor n/p\rfloor-\Big(n^{-1}\sum_{m_n<p\leq n}\log p\, \lfloor n/p\rfloor\Big)^2
\end{split}
\end{equation}
which is a counterpart of \eqref{aux1006}. Observe that the last equality follows from the fact that $\lfloor n/(pq)\rfloor=0$ whenever $p>q>m_n>n^{1/2}$. We claim that
\begin{equation}\label{aux1007}
n^{-1}\sum_{m_n< p\leq n}\log^2 p\, \lfloor n/p \rfloor~=~2^{-1}\log^2 n-2^{-1}\log^2 m_n+o(\log^2 n-\log^2 m_n)
\end{equation}
and
\begin{equation}\label{aux1008}
\Big(n^{-1}\sum_{m_n< p\leq n}\log p\, \lfloor n/p\rfloor\Big)^2=\log^2(n/m_n)+o(\log^{2}(n/m_n)).
\end{equation}
To prove \eqref{aux1007}, write
\begin{align*}
\frac{1}{n}\sum_{m_n< p\leq n}\log^2 p\, \lfloor n/p \rfloor&=\sum_{m_n<p\leq n}p^{-1} \log^2 p+O(\log n)\\
&=2^{-1}\log^2 n-2^{-1}\log^2 m_n+o(\log^2 n-\log^2 m_n)+O(\log n)\\
&=2^{-1}\log^2 n-2^{-1}\log^2 m_n+o(\log^2 n-\log^2 m_n),
\end{align*}
where the first equality follows from the trivial estimate $n/p-1<\lfloor n/p \rfloor\leq n/p$, the second is a consequence of
\eqref{pnt_integral} with $f(x)=x^{-1}\log^2 x$ and $I_x=(m_n,\,n]$, and the third is implied by $$
0\leq \log n/(\log^2 n-\log^2 m_n)\leq 1/\log (n/m_n)~\to~ 0.$$ Formula \eqref{aux1008} follows along similar lines from
$$
n^{-1}\sum_{m_n< p\leq n}\log p\, \lfloor n/p\rfloor=\sum_{m_n<p\leq n}p^{-1} \log p+O(1)=\log (n/m_n)+o(\log(n/m_n)),
$$
where the second equality results from \eqref{pnt_integral} with $f(x)=x^{-1}\log x$ and the term $O(1)$ is killed by $o(\log (n/m_n))$.

Subtracting \eqref{aux1008} from \eqref{aux1007} and using that $\log^2 (n/m_n)\leq \log^2 n-\log^2 m_n$ we arrive at
$$
{\rm Var}(\log \widetilde{U}^{(n,m_n)})=2^{-1}(\log n-\log m_n)(3\log m_n-\log n)+\hat{R}_n,
$$
where
$$
\lim_{n\to\infty}|\hat{R}_n|/(\log^2 n-\log^2 m_n)=0.
$$
This implies \eqref{aux1009} because
\begin{multline*}
\frac{|\hat{R}_n|}{(\log n-\log m_n)(3\log m_n-\log n)}=\frac{|\hat{R}_n|}{\log^2 n-\log ^2 m_n}\frac{\log n+\log m_n}{3\log m_n-\log n}\\
\leq 4\frac{|\hat{R}_n|}{\log^2 n-\log ^2 m_n}~\to~ 0, 
\end{multline*}
where we have used $n^{1/2}<m_n\leq n$ for large $n$ for the inequality.

\end{proof}

\subsection{The big-oh estimate for the fourth central moment}

The proof of Lemma \ref{lem:fourth_moment} is very similar to the proof of Theorem \ref{lem:log_uniform_moments} but is essentially simpler for it only provides an upper estimate rather than the exact rate of growth.
\begin{lemma}\label{lem:fourth_moment}
Assume that $m_n\to\infty$ and $m_n=o(n)$ as $n\to\infty$. Then
\begin{equation}\label{eq:est2}
\E \Big(\log \widetilde{U}^{(n,m_n)}-\E \log \widetilde{U}^{(n,m_n)}\Big)^4 = O(\log^4 m_n),\quad n\to\infty.
\end{equation}
\end{lemma}
\begin{proof}
We divide the proof into two steps. The purpose of Step 1 is to demonstrate that \eqref{eq:est2} is implied by
\begin{equation}\label{eq:est3}
\E \Big(\sum_{p\leq m_n}\log p\cdot \1_{\{\lambda_p(U^{(n)})\geq 1\}}-\E\sum_{p\leq m_n}\log p\cdot \1_{\{\lambda_p(U^{(n)})\geq 1\}}\Big)^4 = O(\log^4 m_n).
\end{equation}
Step 2 is devoted to showing that \eqref{eq:est3} holds true.

\noindent {\sc Step 1.} From Lemma \ref{lem:log_uniform_moments} with $s=2$ we know that
$$\lim_{n\to\infty}\E \big( \log U^{(n)} - \E \log U^{(n)} \big)^4 = 9.$$
This together with \eqref{eq:k_est} yields
$$\lim_{n\to\infty}\E \big( \log \widetilde U^{(n)} - \E \log \widetilde U^{(n)} \big)^4 = O(1).$$
The latter in combination with Lemma \ref{lem:inter} applied with $s=2$,
$$
X_n = \log \widetilde{U}^{(n)} - \E \log \widetilde{U}^{(n)}
=\sum_{p\leq n}\log p\cdot \1_{\{\lambda_p(U^{(n)})\geq 1\}}-\E\sum_{p\leq n}\log p\cdot \1_{\{\lambda_p(U^{(n)})\geq 1\}}
$$
and
$$
Y_n=\sum_{p\leq m_n}\log p\cdot \1_{\{\lambda_p(U^{(n)})\geq 1\}}-\E \sum_{p\leq m_n}\log p\cdot \1_{\{\lambda_p(U^{(n)})\geq 1\}}
$$
enables us to conclude that \eqref{eq:est2} is implied by \eqref{eq:est3} and
$$
\lim_{n\to\infty} \E \Big(\sum_{p\leq m_n}\log p\cdot \1_{\{\lambda_p(U^{(n)})\geq 1\}}-\E \sum_{p\leq m_n}\log p\cdot \1_{\{\lambda_p(U^{(n)})\geq 1\}}\Big)^4=\infty.
$$
The latter holds true in view of
\begin{multline*}
\E \Big(\sum_{p\leq m_n}\log p\cdot \1_{\{\lambda_p(U^{(n)})\geq 1\}}-\E \sum_{p\leq m_n}\log p\cdot \1_{\{\lambda_p(U^{(n)})\geq 1\}}\Big)^4\\
\geq \Big( {\rm Var}\,\Big(\sum_{p\leq m_n}\log p\cdot \1_{\{\lambda_p(U^{(n)})\geq 1\}}\Big)\Big)^2~\to~\infty. 
\end{multline*}
Here, divergence is justified by \eqref{eq:variance_sum_to_m} when $m_n\leq n^{1/2}$ with $m_n\to\infty$ and by Theorem \ref{lem:variance2} together with $$ {\rm Var}\,\Big(\sum_{p\leq m_n}\log p\cdot \1_{\{\lambda_p(U^{(n)})\geq 1\}}\Big)~\sim~\Var (\log \widetilde{U}^{(n,m_n)})$$ (which holds true by another application of Lemma \ref{lem:inter}) when $m_n> n^{1/2}$ with $m_n=o(n)$.

\noindent {\sc Step 2}. Passing to the proof of \eqref{eq:est3} we first note that, for $a>0$,
\begin{equation}\label{eq:sum_log_a_asymp}
\sum_{p\leq m_n}p^{-1} \log^a p=\int_{[2,\,m_n]}x^{-1}\log^a x\,{\rm d}\pi(x)\sim a^{-1}\log^a m_n, 
\end{equation}
where the asymptotic equivalence follows from \eqref{pnt_integral} with $f(x)=x^{-1}\log^{a} x$. With this at hand, we obtain
\begin{align*}
&\hspace{-0.2cm}\E \Big(\sum_{p\leq m_n}\log p\cdot \1_{\{\lambda_p(U^{(n)})\geq 1\}}\Big)^4\\
&=\sum_{p\leq m_n}\log^4 p\, \lfloor n/p\rfloor n^{-1}\\
&+4\sum_{p\neq q\leq m_n}\log^3 p\log q\times \lfloor n/(pq)\rfloor n^{-1}\\
&+6\sum_{p<q\leq m_n}\log^2 p\log^2 q\times \lfloor n/(pq)\rfloor n^{-1}\\
&+12 \sum_{p\neq q\neq r\leq m_n}\log^2 p\log q\log r\times \lfloor n/(pqr)\rfloor n^{-1}\\
&+24\sum_{p<q<r<s\leq m_n}\log p\log q\log r\log s\times \lfloor n/(pqrs)\rfloor n^{-1}.
\end{align*}
Each summand is $O(\log^4 m_n)$, 
by \eqref{eq:sum_log_a_asymp}. For example, for the fourth summand this is a consequence of
\begin{multline*}
\sum_{p\neq q\neq r\leq m_n}\log^2 p\log q\log r\times \lfloor n/(pqr)\rfloor n^{-1}\leq\sum_{p,q,r\leq m_n}(pqr)^{-1}\log^2 p\log q\log r 
\\=\Big(\sum_{p\leq m_n}p^{-1}\log^2 p 
\Big)\Big(\sum_{q\leq m_n}q^{-1}\log q 
\Big)^2  = O(\log^4 m_n). 
\end{multline*}
Since $\E (X-\E X)^4\leq 8\E X^4$ for any random variable $X$ with $\me X^4<\infty$, \eqref{eq:est3} follows.
\end{proof}

\section{Proof of Theorem \ref{thm:main_result2}}\label{sec:proof_1}

We prove Theorem \ref{thm:main_result2} via the sequence of lemmas. For $n\in\mn$ and $t\geq 0$, put
\begin{equation}\label{eq:z_n_definition}
Z_n(t):=\sum_{p\leq n}\log p \cdot\1_{\{\max_{1\leq k\leq \lfloor m_n t\rfloor}\lambda_p(U_k^{(n)})\geq 1\}}.
\end{equation}

\begin{lemma}\label{lem:red1}
Assume that $m_n\to\infty$ and $m_n=o(n)$ as $n\to\infty$. Then, for all $T>0$, 
$$\E \left(\sup_{t\in [0,\,T]} \big(Y_n(\lfloor m_n t\rfloor) - Z_n(t) \big)\right)=O(m_n^{1/2}),\quad n\to\infty.$$
\end{lemma}
\begin{proof}
Fix any $T>0$. Then, for all $t\in [0,T]$,
\begin{align*}
0 & \leq Y_n(\lfloor m_n t\rfloor) - Z_n(t)\\
&=\sum_{p\leq n} \log p \left(\max_{1\leq k\leq \lfloor m_n t\rfloor}\lambda_p(U_k^{(n)})-\1_{\{\max_{1\leq k\leq \lfloor m_n t\rfloor}\lambda_p(U_k^{(n)})\geq 1\}}\right)\\
&\leq \sum_{p\leq m_n^{1/2}}\log p \max_{1\leq k\leq \lfloor m_n t\rfloor}\lambda_p(U_k^{(n)})\\
&+\sum_{m_n^{1/2}<p\leq n} \log p \max_{1\leq k\leq \lfloor m_n t\rfloor}\lambda_p(U_k^{(n)})\1_{\{\max_{1\leq k\leq \lfloor m_n t\rfloor}\lambda_p(U_k^{(n)})\geq 2\}}\\
&:=I_n+J_n.
\end{align*}
It suffices to check that
\begin{equation}\label{eq:Y1_and_Y3_negligible}
\E I_n=O(m_n^{1/2})\quad\text{and}\quad \E J_n=O(m_n^{1/2}).
\end{equation}
To prove the first relation, write
\begin{align*}
\E I_n 
&= \sum_{p\leq m^{1/2}_n}\log p \cdot \E (\max_{1\leq k\leq \lfloor m_n T\rfloor}\lambda_p(U_k^{(n)}))\\
&\leq \sum_{p\leq m^{1/2}_n}\log p\, \sum_{j\geq 1}\P\Big\{\max_{1\leq k\leq \lfloor m_n T\rfloor}\lambda_p(U_k^{(n)})\geq j\Big\}\\
&= \sum_{p\leq m^{1/2}_n}\log p\, \sum_{j\geq 1}(1-(1-n^{-1}\lfloor n/p^j \rfloor)^{\lfloor m_n T\rfloor})\\
&\leq \sum_{p\leq m^{1/2}_n}\log p\, \sum_{j\geq 1}(1-(1-p^{-j})^{\lfloor m_n T\rfloor}).
\end{align*}
Note that
$$
\sum_{j\geq 1}(1-(1-p^{-j})^{\lfloor m_n T\rfloor})=\E (\max_{1\leq k\leq \lfloor m_n T\rfloor}\mathcal{G}_k(p)),$$
where $(\mathcal{G}_k(p))_{k\in\N}$ are  mutually independent random variables with geometric distribution \eqref{geom}. Denote by $\mathcal{E}_1$, $\mathcal{E}_2,\ldots$ independent copies of a random variable having the exponential distribution of unit mean. Using the distributional equality
$$
\max_{1\leq k\leq \lfloor m_n T\rfloor}\mathcal{G}_k(p)\eqdistr \max_{1\leq k\leq \lfloor m_n T\rfloor}\lfloor \log^{-1}p\, \mathcal{E}_k \rfloor=\lfloor \log^{-1}p \max_{1\leq k\leq \lfloor m_n T\rfloor}\mathcal{E}_k \rfloor
$$
we conclude that
\begin{align*}
\sum_{j\geq 1}(1-(1-p^{-j})^{\lfloor m_n T\rfloor})=\E (\max_{1\leq k\leq \lfloor m_n T\rfloor}\mathcal{G}_k(p))
=\E \lfloor \log^{-1} p \max_{1\leq k\leq \lfloor m_n T\rfloor}\mathcal{E}_k \rfloor\\
\leq \E (\log^{-1} p \max_{1\leq k\leq \lfloor m_n T\rfloor}\mathcal{E}_k)=\log^{-1} p \sum_{k=1}^{\lfloor m_n T\rfloor}k^{-1} \leq \log^{-1} p (1+\log (m_n T)).
\end{align*}
Hence,
$$\E I_n \leq \sum_{p\leq m^{1/2}_n} (1+\log (m_n T))=(1+\log (m_n T))\pi(m_n^{1/2}),$$
where $\pi$ is the prime counting function. An application of \eqref{pnt} proves the first relation in \eqref{eq:Y1_and_Y3_negligible}. 

We are now passing to the second relation in \eqref{eq:Y1_and_Y3_negligible}: 
\begin{align*}
&\hspace{-0.3cm}\E J_n \leq \sum_{m^{1/2}_n< p\leq n}\log p \cdot \E (\max_{1\leq k\leq \lfloor m_n T\rfloor}\lambda_p(U_k^{(n)}))\1_{\{\max_{1\leq k\leq \lfloor m_n T\rfloor}\lambda_p(U_k^{(n)})\geq 2\}}\\
&\leq 2\sum_{m^{1/2}_n< p\leq n}\log p \cdot \sum_{j\geq 2}\P\Big\{ \max_{1\leq k\leq \lfloor m_n T\rfloor}\lambda_p(U_k^{(n)})\geq j\Big\}\\
&= 2\sum_{m^{1/2}_n< p\leq n}\log p\,\sum_{j\geq 2}(1-(1-n^{-1}\lfloor n/p^j \rfloor)^{\lfloor m_n T\rfloor})\\
&\leq 2\sum_{m^{1/2}_n< p\leq n}\log p\, \sum_{j\geq 2}(1-(1-p^{-j})^{\lfloor m_n T\rfloor}).
\end{align*}
Using the inequality $1-(1-x)^{\lfloor m_n T\rfloor}\leq x\lfloor m_n T\rfloor\leq x m_n T$, $x\in [0,\,1]$ we infer
\begin{multline*}
\E J_n \leq 2m_n T\sum_{m^{1/2}_n< p\leq n}\log p\,\sum_{j\geq 2}p^{-j}\\
=2m_n T\sum_{m^{1/2}_n< p\leq n}(p^2-p)^{-1} \log p \leq 4m_n T\sum_{m^{1/2}_n< p}p^{-2}\log p=O(m_n^{1/2})
\end{multline*}
having utilized \eqref{pnt_integral} with $f(x)=x^{-2}\log x$ for the last equality.
\end{proof}

Note that
\begin{align*}
\E Z_n(t)&=\E \Big(\sum_{p\leq n}\log p \cdot\1_{\{\max_{1\leq k\leq \lfloor m_n t\rfloor}\lambda_p(U_k^{(n)})\geq 1\}}\Big)\\
&=\sum_{p\leq n}\log p\, (1-(1-n^{-1}\lfloor n/p \rfloor)^{\lfloor m_n t\rfloor})=c_n(\lfloor m_n t\rfloor).
\end{align*}
This equality and Lemma \ref{lem:red1} demonstrate that Theorem \ref{thm:main_result2} follows once we can show that 
\begin{equation}\label{eq:clt_intermediate}
\Big(\frac{Z_n(t)-\E Z_n(t)}{\sqrt{a_n}}\Big)_{t\geq 0} \tofd (B(t))_{t\geq 0},
\end{equation}
where
\begin{equation}\label{eq:a_definition}
a_n=\begin{cases}
2^{-1}m_n\log^2 m_n, & \text{ in case (A)},\\
 2^{-1}m_n(\log n-\log m_n)(3\log m_n-\log n),& \text{ in case (B)}.
\end{cases}
\end{equation}
For later use recall that in view of Theorems \ref{lem:variance} and \ref{lem:variance2}
\begin{equation}\label{eq:an}
a_n \sim m_n \Var \big( \log \widetilde U^{n,m_n}\big).
\end{equation}
As a preparation for the proof of \eqref{eq:clt_intermediate} we need a couple of lemmas. 
\begin{lemma}\label{red3}
Assume that $m_n\to\infty$ and $m_n=o(n)$ as $n\to\infty$. Then
\begin{equation}\label{eq:clt_intermediate_step1}
\Var \Big(\sum_{p\leq m_n}\log p \cdot \1_{\{\max_{1\leq k\leq m_n}\lambda_p(U_k^{(n)})\geq 1\}}\Big)=O(m_n\log m_n),\quad n\to\infty.
\end{equation}
\end{lemma}
\begin{proof}
We start with
\begin{align*}
&\Var \Big(\sum_{p\leq m_n}\log p \cdot \1_{\{\max_{1\leq k\leq m_n}\lambda_p(U_k^{(n)})\geq 1\}}\Big)\leq \sum_{p\leq m_n}\log^2 p\\
&+2\sum_{p<q\leq m_n}\log p\log  q \cdot \Cov \Big(\1_{\{\max_{1\leq k\leq m_n}\lambda_p(U_k^{(n)})\geq 1\}},\1_{\{\max_{1\leq k\leq m_n }\lambda_q(U_k^{(n)})\geq 1\}}\Big).
\end{align*}
In view of \eqref{pnt_integral} with $f(x)=\log^2 x$ the first term is asymptotically equivalent to $m_n\log m_n$.
Further,
\begin{align*}
&\hspace{-0.3cm}\Big|\Cov \Big(\1_{\{\max_{1\leq k\leq m_n }\lambda_p(U_k^{(n)})\geq 1\}},\1_{\{\max_{1\leq k\leq m_n}\lambda_q(U_k^{(n)})\geq 1\}}\Big)\Big|\\
&=\Big|\Cov \Big(\1_{\{\max_{1\leq k\leq m_n }\lambda_p(U_k^{(n)})=0\}},\1_{\{\max_{1\leq k\leq m_n }\lambda_q(U_k^{(n)})=0\}}\Big)\Big|\\
&=\Big|\P^{m_n}\{\lambda_p(U^{(n)})=0,\lambda_q(U^{(n)})=0\}-\P^{m_n}\{\lambda_p(U^{(n)})=0\}\P^{m_n}\{\lambda_q(U^{(n)})=0\}\Big|\\
&\leq m_n\Big|\P\{\lambda_p(U^{(n)})=0,\lambda_q(U^{(n)})=0\}-\P\{\lambda_p(U^{(n)})=0\}\P\{\lambda_q(U^{(n)})=0\}\Big|\\
&=m_n|n^{-1}\lfloor n/(pq)\rfloor-n^{-2}\lfloor n/p \rfloor \lfloor n/q \rfloor|,
\end{align*}
where the equalities
\begin{align}
\P\{\lambda_p(U^{(n)})=0\}&=1-n^{-1}\lfloor n/p \rfloor,\label{eq:divisibility_probabilities1}\\
\P\{\lambda_p(U^{(n)})=0,\lambda_q(U^{(n)})=0\}&=1-n^{-1}\lfloor n/p\rfloor-n^{-1} \lfloor n/q\rfloor+n^{-1}\lfloor n/(pq)\rfloor\label{eq:divisibility_probabilities2}
\end{align}
which hold for prime $p\neq q$ have been utilized. For later use, we note that whenever $q>p>n^{1/2}$ we have
\begin{eqnarray}\label{aux1013}
&&\Cov \Big(\1_{\{\max_{1\leq k\leq m_n }\lambda_p(U_k^{(n)})\geq 1\}},\1_{\{\max_{1\leq k\leq m_n}\lambda_q(U_k^{(n)})\geq 1\}}\Big)\notag\\&=&(1-n^{-1}\lfloor n/p\rfloor-n^{-1} \lfloor n/q\rfloor)^{m_n}\\&-&(1-n^{-1}\lfloor n/p\rfloor-n^{-1} \lfloor n/q\rfloor+n^{-2}\lfloor n/p\rfloor \lfloor n/q\rfloor)^{m_n}\leq 0\notag
\end{eqnarray}
as a consequence of $\lfloor n/(pq)\rfloor=0$. The inequalities $$|n^{-1}\lfloor n/(pq)\rfloor-n^{-2}\lfloor n/p \rfloor \lfloor n/q \rfloor|\leq n^{-1}$$ readily imply that
\begin{multline*}
2\Big|\sum_{p<q\leq m_n}\log p\log  q\, \Cov \Big(\1_{\{\max_{1\leq k\leq m_n}\lambda_p(U_k^{(n)})\geq 1\}}, \1_{\{\max_{1\leq k\leq m_n }\lambda_q(U_k^{(n)})\geq 1\}}\Big)\Big|\\
\leq 2n^{-1} m_n \sum_{p<q\leq m_n}\log p\log q\leq n^{-1} m_n \Big(\sum_{p\leq m_n}\log p\Big)^2 ~\sim~ n^{-1} m_n^3,
\end{multline*}
where the last asymptotic relation is a consequence of \eqref{pnt_integral} with $f(x)=\log x$. Thus, if $m_n\leq n^{1/2}$ and $m_n\to\infty$, the proof of \eqref{eq:clt_intermediate_step1} is complete, for the right-hand side of the last centered formula is $O(m_n)$.

Assume that $m_n>n^{1/2}$ and $m_n=o(n)$. We shall use a representation
\begin{multline*}
\sum_{p\leq m_n}\log p\cdot \1_{\{\max_{1\leq k\leq m_n}\lambda_p(U_k^{(n)})\geq 1\}}=
\sum_{p\leq n^{1/2}}\log p\cdot \1_{\{\max_{1\leq k\leq m_n}\lambda_p(U_k^{(n)})\geq 1\}}\\+
\sum_{n^{1/2}<p\leq m_n}\log p\cdot \1_{\{\max_{1\leq k\leq m_n}\lambda_p(U_k^{(n)})\geq 1\}}.
\end{multline*}
Repeating verbatim the previous argument yields
$$
\Var \left(\sum_{p\leq n^{1/2}}\log p\cdot \1_{\{\max_{1\leq k\leq m_n}\lambda_p(U_k^{(n)})\geq 1\}}\right)=O(n^{1/2}\log n)=O(m_n\log m_n).
$$
Finally, we infer with the help of \eqref{aux1013} that
\begin{multline*}
0\leq \Var\Big(\sum_{n^{1/2}<p\leq m_n}\log p\cdot \1_{\{\max_{1\leq k\leq m_n}\lambda_p(U_k^{(n)})\geq 1\}}\Big)\\ \leq \sum_{n^{1/2}<p\leq m_n}\log^2 p \, \Var (\1_{\{\max_{1\leq k\leq m_n}\lambda_p(U_k^{(n)})\geq 1\}})\leq \sum_{p\leq m_n}\log^2 p=O(m_n\log m_n).
\end{multline*}
\end{proof}

Observe that $a_n/(m_n\log m_n)\to\infty$. In case (A) this is obvious. In case (B) this is a consequence of
\begin{multline*}
a_n/(m_n\log m_n)=2^{-1}(\log n-\log m_n)(3\log m_n-\log n)/\log m_n\\ \geq 2^{-1}(\log n-\log m_n)\to\infty.
\end{multline*}
Thus, using \eqref{eq:clt_intermediate_step1} in combination with Chebyshev's inequality we conclude that \eqref{eq:clt_intermediate} is equivalent to
\begin{multline}\label{eq:clt_intermediate2}
\left(a_n^{-\frac 12}\!\!\sum_{m_n<p\leq n}\!\!\!\log p\, (\1_{\{\max_{1\leq k\leq \lfloor m_n t\rfloor }\lambda_p(U_k^{(n)})\geq 1\}}-\mmp\{\max_{1\leq k\leq \lfloor m_n t\rfloor }\lambda_p(U_k^{(n)})\geq 1\})\right)_{t\geq 0}\\ \tofd (B(t))_{t\geq 0}.
\end{multline}

\begin{lemma}\label{red4}
Assume that $m_n\to\infty$ and $m_n=o(n)$ as $n\to\infty$. Then
\begin{multline}\label{eq:clt_intermediate_step2}
\Var \Big(\sum_{m_n<p\leq n}\log p \Big(\sum_{k=1}^{m_n} \1_{\{\lambda_p(U_k^{(n)})\geq 1\}}-\1_{\{\max_{1\leq k\leq m_n }\lambda_p(U_k^{(n)})\geq 1\}}\Big)\Big)\\
=O(m_n\log m_n), \quad n\to\infty.
\end{multline}
\end{lemma}

For the proof we need a technical result.
\begin{lemma}\label{lem:binomial}
Let ${\rm Bin}(m,\theta)$  be a random variable having a binomial distribution with parameters $m\in\N$ and $\theta \in (0,\,1)$, that is,
$$\mmp\{{\rm Bin}(m,\theta)=k\}=\binom{m}{k}\theta^k(1-\theta)^{m-k},\quad k=0,1,\ldots, m.$$ Then
$$\Var (({\rm Bin}(m,\theta)-1)_+)=\Var ({\rm Bin}(m,\theta)-\1_{\{{\rm Bin}(m,\theta)\geq 1\}})\leq (m\theta)^2.$$
\end{lemma}
\begin{proof}
This follows from
\begin{align*}
&\hspace{-0.3cm}\Var ({\rm Bin}(m,\theta)-\1_{\{{\rm Bin}(m,\theta)\geq 1\}})\\
&=\Var ({\rm Bin}(m,\theta))-2\Cov({\rm Bin}(m,\theta),\1_{\{{\rm Bin}(m,\theta)\geq 1\}})+\Var (\1_{\{{\rm Bin}(m,\theta)\geq 1\}})\\
&=m\theta(1-\theta) -2m\theta(1-\theta)^m
+(1-\theta)^m(1-(1-\theta)^m)\\
&\leq m\theta(1-\theta)-2m\theta(1-\theta)^m+m\theta(1-\theta)^m\\
&=m\theta(1-(1-\theta)^m)-m\theta^2\leq (m\theta)^2.
\end{align*}
Here, we have used twice the inequality $1-(1-\theta)^m\leq \theta m$. 
\end{proof}

\begin{proof}[Proof of Lemma \ref{red4}]
In view of $$\sum_{k=1}^{m_n}\1_{\{\lambda_p(U_k^{(n)})\geq 1\}}-\1_{\{\max_{1\leq k\leq m_n }\lambda_p(U_k^{(n)})\geq 1\}}=\left(\sum_{k=1}^{m_n}\1_{\{\lambda_p(U_k^{(n)})\geq 1\}}-1\right)_+,$$ relation \eqref{eq:clt_intermediate_step2} is equivalent to
$$ \Var \Big(\sum_{m_n<p\leq n}\log p \Big(\sum_{k=1}^{m_n}\1_{\{\lambda_p(U_k^{(n)})\geq 1\}}-1\Big)_+\Big)=O(m_n\log m_n). 
$$
We represent the left-hand side as follows:
\begin{multline*}
\sum_{m_n<p\leq n}\log^2 p\cdot \Var \Big(\Big(\sum_{k=1}^{m_n}\1_{\{\lambda_p(U_k^{(n)})\geq 1\}}-1\Big)_+\Big)+2\sum_{m_n<p<q\leq n}\log p\log q\times \\
\Cov \Big(\Big(\sum_{k=1}^{m_n}\1_{\{\lambda_p(U_k^{(n)})\geq 1\}}-1\Big)_+,\Big(\sum_{k=1}^{m_n}\1_{\{\lambda_q(U_k^{(n)})\geq 1\}}-1\Big)_+\Big)=:A_1(n)+A_2(n).
\end{multline*}
Since the sum $\sum_{k=1}^{m_n}\1_{\{\lambda_p(U_k^{(n)})\geq 1\}}$ has the binomial distribution with parameters $m_n$ and $\P\{\lambda_p(U_k^{(n)})\geq 1\}=n^{-1}\lfloor n/p \rfloor$, an application of Lemma \ref{lem:binomial} gives
\begin{equation}\label{eq:A_1_estimate}
A_1(n)\leq \sum_{p>m_n}\log^2 p\, (n^{-1} m_n \lfloor n/p \rfloor)^2 \leq m_n^2 \sum_{p>m_n}p^{-2} \log^2 p=O(m_n\log m_n),
\end{equation}
where we have used \eqref{pnt_integral} with $f(x)=x^{-2}\log^2 x$ for the last step.  

Passing to the analysis of $A_2(n)$ we have to estimate the covariance: for prime $p\neq q$,
\begin{align*}
C_n(p,q)&:=\Cov \Big(\Big(\sum_{k=1}^{m_n}\1_{\{\lambda_p(U_k^{(n)})\geq 1\}}-1\Big)_+,\Big(\sum_{k=1}^{m_n}\1_{\{\lambda_q(U_k^{(n)})\geq 1\}}-1\Big)_+\Big)\\
&=\sum_{i,j=1}^{m_n}\Cov \Big(\1_{\{\lambda_p(U_i^{(n)})\geq 1\}},\1_{\{\lambda_q(U_j^{(n)})\geq 1\}}\Big)\\
&-\sum_{i=1}^{m_n}\Cov\Big(\1_{\{\max_{1\leq k\leq m_n }\lambda_p(U_k^{(n)})\geq 1\}},\1_{\{\lambda_q(U_i^{(n)})\geq 1\}}\Big)\\
&-\sum_{j=1}^{m_n}\Cov\Big(\1_{\{\max_{1\leq k\leq m_n }\lambda_q(U_k^{(n)})\geq 1\}},\1_{\{\lambda_p(U_j^{(n)})\geq 1\}}\Big)\\
&+\Cov \Big(\1_{\{\max_{1\leq k\leq m_n }\lambda_p(U_k^{(n)})\geq 1\}},\1_{\{\max_{1\leq k\leq m_n }\lambda_q(U_k^{(n)})\geq 1\}}\Big)\\
&=\sum_{i,j=1}^{m_n}\Cov \Big(\1_{\{\lambda_p(U_i^{(n)})\geq 1\}},\1_{\{\lambda_q(U_j^{(n)})\geq 1\}}\Big)\\
&+\sum_{i=1}^{m_n}\Cov\Big(\1_{\{\max_{1\leq k\leq m_n }\lambda_p(U_k^{(n)})=0\}},\1_{\{\lambda_q(U_i^{(n)})\geq 1\}}\Big)\\
&+\sum_{j=1}^{m_n}\Cov\Big(\1_{\{\max_{1\leq k\leq m_n }\lambda_q(U_k^{(n)})=0\}},\1_{\{\lambda_p(U_j^{(n)})\geq 1\}}\Big)\\
&+\Cov \Big(\1_{\{\max_{1\leq k\leq m_n }\lambda_p(U_k^{(n)})=0\}},\1_{\{\max_{1\leq k\leq m_n }\lambda_q(U_k^{(n)})=0\}}\Big).
\end{align*}
For typographical simplicity we shall use the following abbreviations until the end of the proof of \eqref{eq:clt_intermediate_step2}:
$$ a:=n^{-1}\lfloor n/p\rfloor,\quad b:=n^{-1}\lfloor n/q\rfloor,\quad \Delta:=n^{-1}\lfloor n/(pq)\rfloor-n^{-2}\lfloor n/p\rfloor\lfloor n/q\rfloor.$$
Using \eqref{eq:divisibility_probabilities1} and \eqref{eq:divisibility_probabilities2} we obtain
\begin{align*}
C_n(p,q)&=m_n\Delta-m_n\Delta(1-a)^{m_n-1}-m_n\Delta(1-b)^{m_n-1}\\
&+((1-a)(1-b)+\Delta)^{m_n}-(1-a)^{m_n}(1-b)^{m_n}\\
&=m_n\Delta(1-(1-a)^{m_n-1})(1-(1-b)^{m_n-1}))\\
&+\sum_{k=2}^{m_n}\binom{m_n}{k}\Delta^k (1-a)^{m_n-k}(1-b)^{m_n-k}.
\end{align*}
Therefore,
\begin{align}
|A_2(n)|&=\Big| 2\sum_{m_n<p<q\leq n}\log p\log q \cdot C_n(p,q)\Big|\notag\\
& \leq 2m_n\sum_{m_n<p<q\leq n}|\Delta| \log p\log q\cdot (1-(1-a)^{m_n-1})(1-(1-b)^{m_n-1}))\label{eq:A2_estimate}\\
&+2\sum_{k=2}^{m_n}\binom{m_n}{k}\sum_{m_n<p<q\leq n}\log p \log q |\Delta|^k\notag.
\end{align}
The summands are bounded from above, respectively, by 
\begin{equation*}
2m_n^3\sum_{m_n<p<q\leq n}p^{-2}q^{-2}\log p\log q \leq m_n^3 \Big(\sum_{p>m_n} p^{-2}\log p \Big)^2=O(m_n)
\end{equation*}
and
\begin{align*}
&2\sum_{k=2}^{m_n}\binom{m_n}{k}\sum_{m_n<p<q\leq n}p^{-k}q^{-k} \log p \log q \leq \sum_{k=2}^{m_n}\binom{m_n}{k}\Big(\sum_{p>m_n}p^{-k} \log p\Big)^2\\
&\leq {\rm const}\times \sum_{k=2}^{m_n}\binom{m_n}{k}m_n^{2(1-k)}={\rm const}\times m_n^2 ((1+m_n^{-2})^{m_n}-1-m_n^{-1}) 
=O(1),
\end{align*}
where we have used $|\Delta|\leq (pq)^{-1}$, $a<p^{-1}$, $b<q^{-1}$ and the fact that
$$
\sum_{p>m_n}p^{-k} \log p\leq Cm_n^{1-k},
$$
for some constant $C>0$ which does not depend on $k$. The last estimate is justified by Lemma 7.1 in \cite{AlsKabMar:19}.
\end{proof}

Now we conclude with the help of \eqref{eq:clt_intermediate_step2}, $a_n/(m_n\log m_n)\to\infty$ and Chebyshev's inequality that \eqref{eq:clt_intermediate2} is equivalent to \eqref{eq:clt_intermediate3}.

For $1\leq m<n$ and $k\in\N$, put 
$$
\widetilde{U}^{(n)}_k=\prod_{p\leq n}p^{\1_{\{\lambda_p(U_k^{(n)})\geq 1\}}}\quad\text{and}\quad\widetilde{U}_k^{(n,m)}=\prod_{m<p\leq n}p^{\1_{\{\lambda_p(U_k^{(n)})\geq 1\}}}.
$$
\begin{lemma}\label{lem:limit}
Assume that $m_n\to\infty$ and $m_n=o(n)$ as $n\to\infty$. Then
\begin{equation}\label{eq:clt_intermediate3}
\bigg(a_n^{-1/2}\sum_{k=1}^{\lfloor m_n t\rfloor }\Big(\log \widetilde{U}_k^{(n,m_n)}-\E \log \widetilde{U}_k^{(n,m_n)}\Big)\bigg)_{t\geq 0} \\ \tofd (B(t))_{t\geq 0}
\end{equation}
with $a_n$ defined by \eqref{eq:a_definition}.
\end{lemma}
\begin{proof}
The left-hand side in \eqref{eq:clt_intermediate3} is the sum of independent random variables. Thus, 
the proof of \eqref{eq:clt_intermediate3} boils down to showing convergence of covariances
\begin{equation}\label{eq:covar}
a_n^{-1}\Cov \Big(\sum_{k=1}^{\lfloor m_n t\rfloor }\log \widetilde{U}^{(n,m_n)}_k,\sum_{k=1}^{\lfloor m_n s\rfloor }\log \widetilde{U}^{(n,m_n)}_k\Big)~\to~ \min (s,t) 
\end{equation}
for $s,t\geq 0$ and checking the Lindeberg--Feller condition, see Theorem 4.12 in \cite{Kallenberg:1997},
\begin{multline}\label{LF}
\E(\log \widetilde{U}^{(n,m_n)}_1-\E\log \widetilde{U}^{(n,m_n)}_1)^2\1_{\{|\log \widetilde{U}^{(n,m_n)}_1-\E\log \widetilde{U}^{(n,m_n)}_1|>\varepsilon \sqrt{a_n}\}}\\
=o(a_n/m_n)
\end{multline}
for all $\varepsilon>0$. Formula \eqref{eq:covar} follows from
\begin{multline*}
\Cov \Big(\sum_{k=1}^{\lfloor m_n t\rfloor }\log \widetilde{U}^{(n,m_n)}_k,\sum_{k=1}^{\lfloor m_n s\rfloor }\log \widetilde{U}^{(n,m_n)}_k\Big)=\Var\Big(\sum_{k=1}^{\lfloor m_n \min(t,s)\rfloor }\log \widetilde{U}^{(n,m_n)}_k\Big)\\
=\lfloor m_n \min(t,s)\rfloor \Var (\log \widetilde{U}^{(n,m_n)}_1)~\sim~\min(t,s)a_n. 
\end{multline*}
Further, we use the Cauchy--Schwartz inequality and the Markov inequality to obtain
\begin{align*}
&\hspace{-0.4cm}\E(\log \widetilde{U}^{(n,m_n)}_1-\E\log \widetilde{U}^{(n,m_n)}_1)^2\1_{\{|\log \widetilde{U}^{(n,m_n)}_1-\E\log \widetilde{U}^{(n,m_n)}_1|>\varepsilon \sqrt{a_n}\}}\\
&\leq \Big(\E(\log \widetilde{U}^{(n,m_n)}_1-\E\log \widetilde{U}^{(n,m_n)}_1)^4\Big)^{1/2}\\
&\hspace{2cm}\times \Big(\P\{|\log \widetilde{U}^{(n,m_n)}_1-\E\log \widetilde{U}^{(n,m_n)}_1|>\varepsilon \sqrt{a_n}\}\Big)^{1/2}\\
&\leq \Big(\E(\log \widetilde{U}^{(n,m_n)}_1-\E\log \widetilde{U}^{(n,m_n)}_1)^4\Big)^{1/2}\varepsilon^{-1} a_n^{-1/2}(\Var(\log \widetilde{U}^{(n,m_n)}_1))^{1/2}\\&=O(m_n^{-1/2}\log^2 m_n)=o(1)
\end{align*}
which proves \eqref{LF} because $a_n/m_n\to\infty$. Here, the next to the last equality is justified by Lemma \ref{lem:fourth_moment} and \eqref{eq:an} .

We note in passing that \eqref{LF} holds trivially whenever
\begin{equation}\label{limi}
\lim_{n\to\infty}(a_n/\log^2 n)=\infty
\end{equation}
which is particularly the case when $m_n$ grows faster than $\log^2 n$. Observe that $|\log \widetilde{U}^{(n,m_n)}_1-\me \log \widetilde{U}^{(n,m_n)}_1|\leq 2\log n$ a.s.\ as a consequence of $\log \widetilde{U}^{(n,m_n)}_1\leq \log U_1^{(n)}\leq \log n$ a.s. Thus, under \eqref{limi}, the indicator in \eqref{LF} is equal to $0$ for large $n$, whence \eqref{LF}.
\end{proof}

The proof of Theorem \ref{thm:main_result2} is complete.

\section{Proof of Theorem \ref{thm:main_result1}}\label{sec:proof_2}

\noindent {\sc Proof of part (i).} 
We start by noting that whenever $m_n=o(n^{1/2})$, Theorem \ref{thm:main_result2}  immediately implies Theorem \ref{thm:main_result1} because
\eqref{eq:tau_to_m_convergence} yields
$$
\lim_{n\to\infty}\P\{\log\LCM(B_n(\lfloor m_n t\rfloor))=Y_n(\lfloor m_n t\rfloor)\}=1
$$
for each $t\geq 0$.

In the situation that $m_n\neq o(n^{1/2})$ our proof relies on the following inequality
\begin{equation}\label{eq:lipschitz}
|Y_n(k)-Y_n(l)|\leq |k-l|\log n \quad \text{a.s.}
\end{equation}
To prove it, we assume, without loss of generality, that $k>l$. Using the fact that, for
finite sets $A,B\subset \N$,
$$
\LCM(A\cup B)\leq \LCM(A)\times \LCM(B),
$$
we conclude that
\begin{multline*}
0\leq Y_n(k)-Y_n(l)=\log \LCM(U^{(n)}_1,U^{(n)}_2,\ldots,U^{(n)}_k)-\log \LCM(U^{(n)}_1,U^{(n)}_2,\ldots,U^{(n)}_l)\\
\leq \log \LCM(U^{(n)}_{l+1},U^{(n)}_2,\ldots,U^{(n)}_k)\leq \log \prod_{j=l+1}^k U_j^{(n)}\leq (k-l)\log n.
\end{multline*}

Applying \eqref{eq:lipschitz} with $k=\tau^{(n)}(\lfloor m_n t\rfloor)$ and $l=\lfloor m_n t\rfloor$ (note that $k\geq l$ in this case) and taking expectations we arrive at
\begin{multline*}
\frac{\E (Y_n(\tau^{(n)}(\lfloor m_n t\rfloor))-Y_n(\lfloor m_n t\rfloor))}{\sqrt{m_n}\log m_n}\leq \frac{\E (\tau^{(n)}(\lfloor m_n t\rfloor)-\lfloor m_n t\rfloor)\log n}{\sqrt{m_n}\log m_n}\\
=O\Big(\frac{m_n^{3/2}\log n}{n \log m_n}\Big),
\end{multline*}
where the last equality follows from \eqref{eq:tau_expansion0}. It remains to note that in case (A) $m_n\leq n^{1/2}$ for all sufficiently large $n$ and $m_n\to\infty$, so that the right-hand side of the last centered formula converges to zero in view of
$$\frac{m_n^{3/2}\log n}{n\log m_n}\leq n^{-1/4}\log n.$$
Thus,
$$
0\leq \frac{Y_n(\tau^{(n)}(\lfloor m_n t\rfloor))-Y_n(\lfloor m_n t\rfloor)}{\sqrt{m_n}\log m_n}\toprobab 0
$$
which in combination with part (i) of Theorem \ref{thm:main_result2} 
proves  part (i) of Theorem \ref{thm:main_result1}.

\noindent {\sc Proof of part (ii)}.
We appeal to \eqref{eq:lipschitz} once again but now with $k=\tau^{(n)}(\lfloor m_n t\rfloor)$ and $l=\lfloor \E \tau^{(n)}(\lfloor m_n t\rfloor) \rfloor$. This gives
\begin{align*}
&\hspace{-1cm}\E |Y_n(\tau^{(n)}(\lfloor m_n t\rfloor))-Y_n(\lfloor \E \tau^{(n)}(\lfloor m_n t\rfloor) \rfloor)|\\
&\leq \log n\, \E |\tau^{(n)}(\lfloor m_n t\rfloor)-\lfloor \E \tau^{(n)}(\lfloor m_n t\rfloor) \rfloor|\\
&\leq \log n + \log n (\Var ( \tau^{(n)}(\lfloor m_n t\rfloor)))^{1/2}=O(n^{-1/2} m_n\log n),
\end{align*}
where we have used \eqref{eq:tau_expansion0} and $m_n>n^{1/2}$ for all large enough $n$. Therefore,
$$
a_n^{-\frac 12}\big(Y_n(\tau^{(n)}(\lfloor m_n t\rfloor))-Y_n(\lfloor \E \tau^{(n)}(\lfloor m_n t\rfloor) \rfloor)\big) \toprobab 0,
$$
by Markov's inequality as a consequence of
$$ \frac{m_n^2\log^2n}{n a_n} =  \frac{2 m_n\log^2 n}{n (\log n-\log m_n)(3\log m_n-\log n)}\leq \frac{4 m_n \log n}{n}\frac{1}{\log(n/m_n)}.$$
The first term is bounded because $m_n=O(n(\log n)^{-1})$ by assumption and the second tends to zero in view of $m_n=o(n)$.

It remains to prove the convergence
\begin{equation}\label{eq:CLT_case_B_proof1}
\Big(a_n^{-1/2}\big(Y_n(\lfloor \E \tau^{(n)}(\lfloor m_n t\rfloor) \rfloor)-c_n(-n\log (1-(m_n t)/n)))\big)\Big)_{t\geq 0}\tofd (B(t))_{t\geq 0}.
\end{equation}
Since $\lfloor \E \tau^{(n)}(\lfloor m_n t\rfloor) \rfloor=O(m_n)$, Lemma  \ref{lem:red1} implies that
$$
\E \Big(\sup_{t\in [0,\,T]}\Big(Y_n(\lfloor \E \tau^{(n)}(\lfloor m_n t\rfloor) \rfloor)-Z_n(m_n^{-1}\E \tau^{(n)}(\lfloor m_n t\rfloor))\Big)\Big)=O(m_n^{1/2})
$$
for every fixed $T>0$, where the definition of $Z_n$, 
see \eqref{eq:z_n_definition}, has to be recalled.

Thus, \eqref{eq:CLT_case_B_proof1} follows if we can check
\begin{equation}\label{eq:CLT_case_B_proof2}
\Big(a_n^{-1/2} \big( Z_n(m_n^{-1}\E \tau^{(n)}(\lfloor m_n t\rfloor))-c_n(-n\log (1-(m_n t)/n)))\big) \Big)_{t\geq 0}\tofd (B(t))_{t\geq 0}.
\end{equation}
For fixed $n$, the function $c_n$ defined in \eqref{eq:cn} is increasing and subadditive on $(0,\infty)$, whence 
\begin{multline*}
0\leq c_n(x+y)-c_n(x) \leq c_n(y) \leq \sum_{p\leq n} \log p (1-(1-n^{-1}\lfloor n/p \rfloor)^y)\\
\leq y\sum_{p\leq n}p^{-1}\log p \leq {\rm const}\times y\log n
\end{multline*}
and, see \eqref{eq:tau_expectation},
$$
\lfloor \E \tau^{(n)}(\lfloor m_n t\rfloor)\rfloor =\lfloor n(H_n-
H_{n- \lfloor m_nt\rfloor })\rfloor
=-n\log(1-(m_n t)/n)+O(1). 
$$
Therefore, the equality $\E Z_n(t)=c_n(\lfloor m_n t\rfloor)$ shows that convergence \eqref{eq:CLT_case_B_proof2} is equivalent to
\begin{equation}\label{eq:CLT_case_B_proof3}
\left( a_n^{-\frac 12} \big(Z_n(m_n^{-1}\E \tau^{(n)}(\lfloor m_n t\rfloor))-\E Z_n(m_n^{-1}\E \tau^{(n)}(\lfloor m_n t\rfloor))\big)\right)_{t\geq 0}\tofd (B(t))_{t\geq 0}.
\end{equation}
Using once again the estimate $\lfloor \E \tau^{(n)}(\lfloor m_n t\rfloor) \rfloor=O(m_n)$, we can apply Lemmas \ref{red3} and \ref{red4} to deduce that \eqref{eq:CLT_case_B_proof3} is implied by
\begin{equation}\label{eq:CLT_case_B_proof4}
\Big( a_n^{-1/2} \times \sum_{k=1}^{\lfloor \E \tau^{(n)}(\lfloor m_n t\rfloor)\rfloor}(\log \widetilde{U}_k^{(n,m_n)}-\E\log \widetilde{U}_k^{(n,m_n)})\Big)_{t\geq 0}\tofd (B(t))_{t\geq 0}.
\end{equation}
The latter relation follows from Lemma \ref{lem:limit} and 
\begin{multline*}
\Var \Big(\sum_{k=\lfloor m_nt \rfloor+1}^{\lfloor \E \tau^{(n)}(\lfloor m_n t\rfloor)\rfloor}(\log \widetilde{U}_k^{(n,m_n)}-\E\log \widetilde{U}_k^{(n,m_n)})\Big)\\
\leq (\E \tau^{(n)}(\lfloor m_n t\rfloor)-\lfloor m_nt \rfloor) \Var (\log \widetilde{U}_1^{(n,m_n)})=O(n^{-1} m_n^2 \Var (\log \widetilde{U}_1^{(n,m_n)}))
\end{multline*}
which, in view of Chebyshev's inequality, \eqref{eq:an} and $m_n=o(n)$, 
ensures 
$$
a_n^{-1/2}\times \sum_{k=\lfloor m_nt \rfloor+1}^{\lfloor \E \tau^{(n)}(\lfloor m_n t\rfloor)\rfloor}(\log \widetilde{U}_k^{(n,m_n)}-\E\log \widetilde{U}_k^{(n,m_n)})\toprobab 0.
$$

\section{Appendix}\label{sec:appendix}

We intend to show that the normalization in a limit theorem for $\widehat{Y}_n(\lfloor m_n t\rfloor)$ defined in \eqref{eq:y_n_geom_def} is different from that in Theorem \ref{thm:main_result2} unless $\lim_{n\to\infty}(\log n/\log m)=1$. We confine ourselves with the one-dimensional convergence.
\begin{proposition}\label{prop:geom}
Assume that $m_n\to\infty$ and $m_n=o(n)$ as $n\to\infty$. Then
\begin{equation}\label{eq:geom_lt}
\frac{\widehat{Y}_n(m_n)-\sum_{p\leq n}\log p(1-(1-p^{-1})^{m_n})}{\sqrt{2^{-1}m_n(\log^2 n-\log^2 m_n)}}~\todistr~ B(1),
\end{equation}
where $B(1)$ has a standard normal distribution.
\end{proposition}
\begin{proof}
Repeating verbatim the argument used in the proof of Lemma \ref{lem:red1} we obtain a counterpart of the limit relation stated in that lemma 
$$\E \sum_{p\leq n} \log p \left(\max_{1\leq k\leq \lfloor m_n t\rfloor}\mathcal{G}_k(p)-\1_{\{\max_{1\leq k\leq \lfloor m_n t\rfloor}\mathcal{G}_k(p)\geq 1\}}\right)=O(m_n^{1/2}).$$
Thus, it is enough to prove \eqref{eq:geom_lt} with $\widehat{Y}_n(m_n)$ replaced by
$$\widehat{Z}_n(1)=\sum_{p\leq n}\log p \cdot\1_{\{\max_{1\leq k\leq m_n }\mathcal{G}_k(p)\geq 1\}}.$$ The centering in \eqref{eq:geom_lt} is just the expectation of $\widehat{Z}_n(1)$. To calculate the variance we argue as follows
\begin{align*}
\Var(Z_n(1))&=\sum_{p\leq n}\log^2 p (1-(1-p^{-1})^{m_n})(1-p^{-1})^{m_n}\\
&=\sum_{m_n<p\leq n}\log^2 p(1-(1-p^{-1})^{m_n})(1-p^{-1})^{m_n}+O(m_n\log m_n)\\
&=\sum_{m_n<p\leq n}\log^2 p(1-(1-p^{-1})^{m_n})+O(m_n\log m_n),
\end{align*}
where we have used \eqref{pnt_integral} to obtain $$\sum_{p\leq m_n}\log^2 p (1-(1-p^{-1})^{m_n})(1-p^{-1})^{m_n}\leq \sum_{p\leq m_n}\log^2 p=O(m_n\log m_n)$$ and $$\sum_{m_n<p\leq n}\log^2 p(1-(1-p^{-1})^{m_n})^2\leq m_n^2\sum_{m_n<p}p^{-2} \log^2 p =O(m_n\log m_n).$$
Further,
\begin{multline*}
m_n\sum_{m_n<p\leq n}p^{-1} \log^2 p -2^{-1} m_n(m_n-1)\sum_{m_n<p\leq n}p^{-2} \log^2 p\\
\leq \sum_{m_n<p\leq n}\log^2 p (1-(1-p^{-1})^{m_n})\leq m_n\sum_{m_n<p\leq n}p^{-1} \log^2 p
\end{multline*}
and $$m_n^2\sum_{m_n<p\leq n}p^{-2} \log^2 p=O(m_n\log m_n)$$ by \eqref{eq:A_1_estimate}. It remains to note that $$m_n\sum_{m_n<p\leq n}p^{-1} \log^2 p~\sim~2^{-1}m_n(\log^2 n-\log^2 m_n) 
$$
by \eqref{pnt_integral} and $$ (m_n\log m_n)/ (m_n(\log^2 n-\log^2 m_n)) \leq 1/ \log (n/m_n)~\to~ 0. 
$$ Put $b_n:=2^{-1}m_n(\log^2 n-\log^2 m_n)$ and $$V_p(n):=\1_{\{\max_{1\leq k\leq m_n}\mathcal{G}_k(p)\geq 1\}}-\mmp\left\{\max_{1\leq k\leq m_n} \mathcal{G}_k(p)\geq 1\right\}$$ for $n\in\mn$ and $p\in\mathcal{P}$. The Lindeberg--Feller condition, see Theorem 4.12 in \cite{Kallenberg:1997},
\begin{equation*}
\sum_{p\leq n}\log^2 p\cdot \me V^2_p (n)\1_{\{\log p|V_p(n)|>\varepsilon \sqrt{b_n}\}}=o(b_n)
\end{equation*}
for all $\varepsilon>0$ is trivial because $\log p\,|V_p(n)| \leq \log n$ for $p\leq n$ and $m_n(\log^2 n-\log^2 m_n)$ grows faster than $\log^2 n$. The latter follows from    
$$ m_n(\log^2 n-\log^2 m_n)/\log^2 n\geq m_n\log(n/m_n)/\log n\geq \log (n/m_n) \to\infty$$ in the case $m_n\geq \lfloor \log n\rfloor$ and $m_n=o(n)$ and $$m_n(\log^2 n-\log^2 m_n)/\log^2 n\geq m_n\log(n/m_n)/\log n\sim m_n \to\infty$$ in the case $m_n< \lfloor \log n\rfloor$ and $m_n\to\infty$.
\end{proof}
Now we use Theorem \ref{thm:main_result2} and Proposition \ref{prop:geom} to conclude that $\lim_{n\to\infty}(b_n/a_n)=1$ if, and only if, $\lim_{n\to\infty}(\log n/\log m_n)=1$. Recall that $a_n$ is defined in \eqref{eq:a_definition}.  If, for instance, $m_n\leq n^{1/2}$ and $m_n\to\infty$, then $a_n=2^{-1}m_n\log^2 m_n$ so that $b_n/a_n\geq 3$. In particular, $\lim_{n\to\infty}(b_n/a_n)=\infty$ whenever $\lim_{n\to\infty}(\log n/\log m)=\infty$.

Here is the promised sketch of the proof of Proposition \ref{prop:fixed_m}.
\begin{proof}[Proof of Proposition \ref{prop:fixed_m}]
It is enough to prove the joint convergence
\begin{multline*}
\Big(\Big(\log \LCM (U_1^{(n)},U_2^{(n)},\ldots, U_m^{(n)})-\sum_{k=1}^m \log U_k^{(n)}\Big)_{m\in\N}, (n^{-1} U_m^{(n)})_{m\in\N}\Big)\todistri\\
\Big(\Big(\sum_{p}\log p\cdot\Big(\max_{1\leq k\leq m}\mathcal{G}_k(p)-\sum_{k=1}^{m}\mathcal{G}_k(p)\Big)\Big)_{m\in\N},(U_m)_{m\in\N}\Big).
\end{multline*}
Note that, for every fixed $M\in\N$,
\begin{align*}
&\hspace{-2cm}\log \LCM (U_1^{(n)},U_2^{(n)},\ldots,U_m^{(n)})-\sum_{k=1}^{m}\log U_k^{(n)}\\
&=\sum_{p\leq n}\log p\cdot \Big(\max_{1\leq k\leq m}\lambda_p(U_{k}^{(n)})-\sum_{k=1}^{m}\lambda_p(U_{k}^{(n)})\Big)\\
&=\sum_{p\leq M}\cdots + \sum_{M<p\leq n}\cdots=:W_1(n,m,M)+W_2(n,m,M).
\end{align*}
By the continuous mapping theorem applied to part (ii) of Lemma \ref{lem:conv_to_geom},
\begin{multline*}
\Big((W_1(n,m,M))_{m\in\N}, (n^{-1} U_m^{(n)})_{m\in\N}\Big)\todistri\\
\Big(\Big(\sum_{p\leq M}\log p\cdot\Big(\max_{1\leq k\leq m}\mathcal{G}_k(p)-\sum_{k=1}^{m}\mathcal{G}_k(p)\Big)\Big)_{m\in\N},(U_m)_{m\in\N}\Big)
\end{multline*}
for every fixed $M\in\N$. Further, using Theorem 2.3(i) in \cite{BosMarRas:19} with $f(x)=x$ yields
\begin{multline*}
\lim_{M\to\infty}\sum_{p\leq M}\log p\cdot\Big(\max_{1\leq k\leq m}\mathcal{G}_k(p)-\sum_{k=1}^{m}\mathcal{G}_k(p)\Big)\\
=\sum_{p}\log p\cdot\Big(\max_{1\leq k\leq m}\mathcal{G}_k(p)-\sum_{k=1}^{m}\mathcal{G}_k(p)\Big)\quad\text{a.s.}
\end{multline*}
It remains to invoke Theorem 3.2 in \cite{Billingsley:1999} in conjunction with the relation
$$
\lim_{M\to\infty}\limsup_{n\to\infty}\P\{|W_2(n,m,M)|>\varepsilon\}=0
$$
for every fixed $\varepsilon>0$ and $m\in\N$. For the latter, see formula (22) in \cite{BosMarRas:19}.
\end{proof}

\section*{Acknowledgments}
The work of AI and AM has received funding from the Ulam Program of the Polish National Agency for Academic Exchange (NAWA), projects no. PPN/ULM/ 2019/1/00010/DEC/1 and PPN/ULM/2019/1/00004/DEC/1, respectively. DB was  partially supported by the National Science Center, Poland (grant number  2019/33/ B/ST1/00207).

\end{document}